\documentclass[12pt]{amsart}
\usepackage{a4wide,enumerate,color,graphicx}
\usepackage[normalem]{ulem}
\usepackage{amsmath}
\usepackage{scrextend} 
\usepackage{tikz}
\allowdisplaybreaks

\let\pa\partial  
\let\na\nabla  
\let\eps\varepsilon  
\newcommand{\N}{{\mathbb N}}  
\newcommand{\R}{{\mathbb R}} 
 
\newcommand{\diver}{\operatorname{div}}  
\newcommand{\dom}{{\mathcal O}}
\newcommand{\DD}{{\mathcal D}}
\newcommand{\E}{{\mathbb E}}
\newcommand{\Prob}{\mathbb{P}}

\newcommand\dela[1]{}

\newtheorem{theorem}{Theorem}   
\newtheorem{lemma}[theorem]{Lemma}   
   
\newtheorem{proposition}[theorem]{Proposition}   
\newtheorem{remark}[theorem]{Remark}   
\newtheorem{corollary}[theorem]{Corollary}  
\newtheorem{definition}{Definition}

\begin{document}  

\title[Global martingale solutions for quasilinear SPDEs]{Global martingale solutions for quasilinear SPDEs \linebreak via the boundedness-by-entropy method} 

\author[G. Dhariwal]{Gaurav Dhariwal}
\address{Institute for Analysis and Scientific Computing, Vienna University of  
	Technology, Wiedner Hauptstra\ss e 8--10, 1040 Wien, Austria}
\email{gaurav.dhariwal@tuwien.ac.at} 

\author[F. Huber]{Florian Huber}
\address{Institute for Analysis and Scientific Computing, Vienna University of  
	Technology, Wiedner Hauptstra\ss e 8--10, 1040 Wien, Austria}
\email{florian.huber@asc@tuwien.ac.at}

\author[A. J\"ungel]{Ansgar J\"ungel}
\address{Institute for Analysis and Scientific Computing, Vienna University of  
	Technology, Wiedner Hauptstra\ss e 8--10, 1040 Wien, Austria}
\email{juengel@tuwien.ac.at} 

\author[C. Kuehn]{Christian Kuehn}
\address{Department of Mathematics, Technical University of Munich, Boltzmannstr. 3,
85748 Garching bei M\"unchen}
\email{ckuehn@ma.tum.de} 

\author[A. Neam\c tu]{Alexandra Neam\c tu}
\address{Department of Mathematics, Technical University of Munich, Boltzmannstr. 3,
85748 Garching bei M\"unchen}
\email{alexandra.neamtu@tum.de} 

\date{\today}


\begin{abstract}
The existence of global-in-time bounded martingale solutions to a general class of
cross-diffusion systems with multiplicative Stratonovich noise is proved. The equations
describe multicomponent systems from physics or biology with volume-filling effects 
and possess a formal gradient-flow or entropy structure. 
This structure allows for the derivation of almost surely 
positive lower and upper bounds for the stochastic processes. 
The existence result holds under some assumptions on the interplay between the
entropy density and the multiplicative noise terms.
The proof is based on a stochastic Galerkin method, a Wong--Zakai type approximation 
of the Wiener process, the boundedness-by-entropy
method, and the tightness criterion of Brze\'{z}niak and coworkers. 
Three-species Maxwell--Stefan systems and $n$-species biofilm models are examples
that satisfy the general assumptions.
\end{abstract}

 

\maketitle


\section{Introduction}

Cross-diffusion systems arise in many application areas like 
fluid dynamics of mixtures, electrochemistry, cell biology, and biofilm modeling. 
Cross diffusion occurs if the gradient in the concentration of 
one {species} induces a flux of another {species}. In many applications, 
{volume-filling} effects need to be taken into account because of the finite size of
the species or components, which means that the unknowns are
volume fractions which sum up to one. Such cross-diffusion systems with
volume filling {in deterministic setting} were analyzed in, for instance, \cite{BrCh14,BDPS10,DGM13} 
in the context of gas mixtures or ion transport through membranes. The
boundedness-by-entropy method \cite{Jue15} provides a framework for the
existence analysis and the proof of positive
lower and upper bounds for the concentrations.
The aim of this paper is to extend this technique to the stochastic setting. 
We prove the global-in-time existence of martingale solutions to cross-diffusion
systems with volume filling and Stratonovich stochastic forcing.


\subsection{Model equations}

The dynamics of the concentration (or volume fraction)
vector $u=(u_1,\ldots,u_n)$ is given by
\begin{equation}\label{1.eq}
  \textnormal{d}u_i - \diver\bigg(\sum_{j=1}^n A_{ij}(u)\na u_j\bigg)\textnormal{d}t
	= \sum_{j=1}^n\sigma_{ij}(u)\circ\textnormal{d}W_j(t)\quad\mbox{in }\dom,\ t>0,
\end{equation}
where $i=1,\ldots,n$ and $\dom\subset\R^d$ ($1\le d\le 3$) is a bounded domain,
supplemented with the no-flux boundary and initial conditions,
\begin{equation}\label{1.bic}
  \sum_{j=1}^n A_{ij}(u)\na u_j\cdot\nu = 0\quad\mbox{on }\pa\dom,\ t>0, \quad
	u_i(0)=u_i^0\quad\mbox{in }\dom,\ i=1,\ldots,n.
\end{equation}
Here, $\nu$ is the exterior
unit normal vector to $\pa\dom$ and $u_i^0$ is a possibly random initial datum.
The concentrations $u_i(\omega,x,t)$
are defined on $\Omega\times\dom\times[0,T]$, where
$\omega\in\Omega$ represents the stochastic variable, $x\in\dom$ the spatial
variable, and $t\in[0,T]$ the time. Together with the solvent concentration
$u_{n+1}$, the concentrations fill up the domain, i.e.,
$\sum_{i=1}^{n+1}u_i=1$. We call this assumption {\em volume filling}.
The matrix $A(u)=(A_{ij}(u))$ is the
diffusion matrix, $\sigma(u)=(\sigma_{ij}(u))$ is the multiplicative noise term,
and $W=(W_1,\ldots,W_n)$ is an $n$-dimensional Wiener process. 
Details on the stochastic framework will be given in Section \ref{sec.main}.
The stochastic forcing represents external perturbations or a lack of knowledge
of certain physical or biological parameters.

Equations \eqref{1.eq} can be equivalently formulated in the It\^o form
\cite[Section 6.5]{Eva13}:
$$
  \textnormal{d}u_i - \diver\bigg(\sum_{j=1}^n A_{ij}(u)\na u_j\bigg)\textnormal{d}t
	= \sum_{j=1}^n\sigma_{ij}(u) \textnormal{d}W_j(t) 
	+ \frac12\bigg(\sum_{k=1}^n\sum_{j=1}^n\sigma_{kj}(u)
	\frac{\pa\sigma_{ij}}{\pa u_k}(u)\bigg)\textnormal{d}t,
$$
where $i=1,\ldots,n$, and this formulation will be also used in our analysis.
The formulation of \eqref{1.eq} in the Stratonovich form comes purely from a modeling viewpoint. In fact, our analysis uses the Wong--Zakai approximation, where we approximate the noise by smooth functions, thus obtaining a system of PDEs, which in {turn} converge in the limit to stochastic differential equations in the Stratonovich form. 
Alternatively, we could consider \eqref{1.eq} in the It\^o form and
include the correction term in the formulation. In fact, considering the It\^o formulation would enable us to treat more general infinite-dimensional noise
{but increasing the already involved technicalities}.
{Therefore,} this aspect will be discussed in a future work. 

Quasilinear stochastic partial differential equations (SPDEs) (e.g.\ the porous-media equation) have been extensively analyzed using the theory of (locally) monotone operators\,\cite{BaPrRo09, Gess12, LiRo10, PrRo07} or approximating the corresponding coefficients by locally monotone ones \cite{HoZh17}.  Recently, there has been a growing interest in developing various solution concepts for quasilinear SPDEs such as: kinetic\,\cite{DeHoVo16, FeGe19, GeHo18}, strong (in the probabilistic sense) and weak (in the PDE sense)\,\cite{DDMH15,HoZh17}, entropy\,\cite{DaGeGe19}, martingale\,\cite{DeHoVo16, DJZ19} or (pathwise) mild solutions\,\cite{KuNe18}. We mention that solution concepts for certain quasilinear SPDEs have been developed also via rough paths theory\,\cite{OtWe19}, paracontrolled calculus\,\cite{BDH19, FuGu19}, or regularity structures\,\cite{GeHa19}.
To our best knowledge, the techniques employed in this context heavily rely on the fact that the diffusion matrix is symmetric and/or positive-(semi)definite\,\cite{DDMH15, DeHoVo16, HoZh17, Hor17}.
However, in many applications, the diffusion matrix does not satisfy these requirements, i.e., it is neither symmetric nor positive
semi-definite. Therefore, most of the techniques available in the literature on quasilinear SPDEs do not apply  
or allow only local-in-time solutions\,\cite{Hor17, KuNe18}. 
The main goal of this work is to prove the {\em global-in-time existence} of martingale solutions for quasilinear SPDEs whose diffusion matrix is {\em neither symmetric nor positive semi-definite, but admits a certain structure}, which we precisely describe below.

It turns out that deterministic cross-diffusion systems arising from (thermodynamic) applications often
have a special structure, a so-called entropy or formal gradient-flow structure,
which can be exploited for the existence analysis. This means that there exists 
an entropy density $h:[0,\infty)^n\to\R$ such that the deterministic analog of 
\eqref{1.eq} can be formulated in terms of the entropy variables
$w_i:=\pa h/\pa u_i$ as
\begin{equation}\label{1.B}
  \pa_t u_i(w) - \diver\bigg(\sum_{j=1}^n B_{ij}(w)\na w_j\bigg) = 0,
	\quad i=1,\ldots,n,
\end{equation}
and the so-called Onsager matrix 
$B(w)=A(u(w))h''(u(w))^{-1}$ is positive semi-definite, where
$h''(u)^{-1}$ denotes the inverse of the Hessian of $h$ and $u=u(w)=(h')^{-1}(w)$
is now interpreted as a vector-valued function of $w$,
assuming that the inverse of $h'$ exists. An example is the Boltzmann-type
entropy density {$h_B(u)=\sum_{i=1}^{n+1}(u_i(\log u_i-1)+1)$}. 
Using $w_i$ as a test function in \eqref{1.B}, a formal computation leads to
\begin{equation}\label{1.ei}
  \frac{\textnormal{d}}{\textnormal{d}t}\int_\dom h(u)\,\textnormal{d}x + \int_\dom\na u:h''(u)A(u)\na u \,\textnormal{d}x = 0,
\end{equation}
where ``:'' denotes the Frobenius matrix product.
Since $B(w)$ is positive semi-definite,
so is $h''(u)A(u)$, and we infer that $t\mapsto\int_\dom h(u(t))\,\textnormal{d}x$ is a Lyapunov
functional along solutions to \eqref{1.B}. 

The volume-filling condition $\sum_{i=1}^{n+1} u_i=1$ implies that the solvent
concentration can be replaced by the other concentrations $u_i\ge 0$ according to 
$u_{n+1}=1-\sum_{i=1}^n u_i$. This means that the concentration vector 
$u=(u_1,\ldots,u_n)$ is an element of the Gibbs simplex
$\DD=\{u\in {(0,1)}^{n}:\sum_{i=1}^n u_i<1\}$. 
If $h$ is invertible on $\DD$, we can define
$u(w)=(h')^{-1}(w)$, and this function maps $\R^n$ to $\DD$. Thus, if $w(x,t)$
is a solution to \eqref{1.B}, $u(w(x,t))\in\DD$ is componentwise
positive and bounded from above.
This provides $L^\infty$ estimates without using
a maximum principle which generally cannot be applied to cross-diffusion
systems. In this paper, we show that this idea can be extended to the
stochastic case, allowing for $L^\infty$ bounds almost surely.

Examples for cross-diffusion systems \eqref{1.eq} with volume filling are
the Maxwell--Stefan equations and certain biofilm models (see Section
\ref{sec.ex} for details). For fluid mixtures with three components, 
the Maxwell--Stefan diffusion matrix equals
\begin{align*}
  & A(u) = \frac{1}{a(u)}\begin{pmatrix} 
	d_2+(d_0-d_2)u_1 & (d_0-d_1)u_1 \\ (d_0-d_2)u_2 & d_1+(d_0-d_1)u_2 \end{pmatrix}, \\
	&\mbox{where }a(u)=d_0d_1u_1 + d_0d_2u_2 + d_1d_2u_3,
\end{align*}
and $d_i>0$ for {$i=0,1,2$}. This matrix is generally non-symmetric and 
not positive definite, but its eigenvalues are positive (this allows for
local smooth deterministic solutions; see \cite{Ama90}).
The first global existence
result for deterministic Maxwell--Stefan equations was proved in \cite{GiMa98} for initial data
around the constant equilibrium state. The existence of local classical solutions
was shown in \cite{Bot11}. The entropy structure was revealed in \cite{JuSt13},
and a general global existence result was proved. 
Other cross-diffusion models with volume filling arise in ion-transport and
biofilm modeling \cite{DMZ19,GeJu18}.
A general class of volume-filling systems was formally 
derived in \cite{ZaJu17} from a random walk on a lattice. 

In the stochastic setting, we need to overcome some technical obstacles. 
First, since the diffusion matrix is not symmetric and not positive definite, standard
semigroup theory is not applicable. Second, the application
of the It\^o formula to derive the stochastic analog of the entropy identity 
\eqref{1.ei} requires that
the entropy density {be} an element of $C^2(\overline{\DD})$ 
which is usually not the case.
For instance, the Boltzmann-type entropy density satisfies 
$\pa^2 h_B/\pa u_i^2=1/u_i+1/u_{n+1}$ which is undefined when $u_i=0$ or $u_{n+1}=0$.
Third, the system \eqref{1.B} is approximated in \cite{Jue15} by the implicit
Euler discretization which is not compatible with the stochastic term
(neither in It\^o nor in Stratonovich form). {We point out} that the implicit Euler discretization, which is implemented in \cite{Jue15}, could be avoided by introducing an additional regularization, hence avoiding the incompatibility issue, but this idea needs to be explored further.  

Our key idea is to approximate the noise by a Wong--Zakai type argument
and the space by a stochastic Galerkin method. 
This results in a system of ordinary differential
equations which can be treated by the boundedness-by-entropy method \cite{Jue15}.
The limit of vanishing Wong--Zakai parameter requires also the existence of
solutions to a Galerkin stochastic differential system. 
This is proved by a fixed-point argument 
up to a stopping time $\tau_R>0$, i.e., up to the first 
time a certain norm 
of the solution is larger than some $R>0$. Estimates uniform in the Galerkin 
dimension $N\in\N$ are derived from an entropy inequality, 
which needs a regularization $h_\delta$ of the entropy density $h$, such that 
$h_\delta\in C^2(\overline{\DD})$ with $\delta>0$.
The final step are the limits $\delta\to 0$, $R\to\infty$, and $N\to\infty$.
Details of this procedure are given in Section \ref{sec.ideas}.


\subsection{Notation and stochastic framework}
\label{sec.notstoch}

Let $\dom\subset\R^d$ ($d\ge 1$) be a bounded domain.
The usual Lebesgue and Sobolev spaces are denoted by $L^p(\dom)$ and
$W^{k,p}(\dom)$, respectively, where $p\in[1,\infty]$, $k\in\N$, and we set
$H^k(\dom)=W^{k,2}(\dom)$.  The norm of a function $u=(u_1,\ldots,u_n)
\in L^2(\dom;\R^n)$ is understood as $\|u\|_{L^2(\dom)}^2
=\sum_{i=1}^n\|u_i\|_{L^2(\dom)}^2$, and we use this notation also for other vector- 
or matrix-valued functions. We write $\langle u,v\rangle$ for the dual
product between $H^3(\dom)'$ and $H^3(\dom)$. We use the same notation 
if $u$, $v\in L^2(\dom)$, and in this case, $\langle u,v\rangle=\int_\dom uv\,\textnormal{d}x$.
In the vector-valued case, we have $\langle u,v\rangle=\sum_{i=1}^n\int_\dom
u_iv_i\,\textnormal{d}x$ for $u$, $v\in L^2(\dom;\R^n)$. 
The set $\DD=\{u\in{(0,1)^n}:\sum_{i=1}^{n}u_i< 1\}$ 
is the Gibbs simplex in $\R^n$, and we set
$u_{n+1}:=1-\sum_{i=1}^n u_i>0$ if $u\in\DD$.

Let $(\Omega,\mathcal{F},\Prob)$ be a probability space endowed with a complete
right-continuous filtration $\mathbb{F}=(\mathcal{F})_{t\ge 0}$ and let $H$ be
a Hilbert space. The space $L^2(\Omega;H)$ consists of all $H$-valued random variables
$u$ such that
$$
  \E\|u\|_H^2 := \int_\Omega\|u(\omega)\|_H^2\Prob(\textnormal{d}\omega) < \infty.
$$
Let $(\tilde{\eta}_k)_{k = 1}^n$ be the canonical basis of $\R^n$. We denote by
$$
  \mathcal{L}_2(\R^n;L^2(\dom)) := \bigg\{L:\R^n\to L^2(\dom)\mbox{ linear continuous: }
	\sum_{k=1}^n\|L\tilde{\eta}_k\|_{L^2(\dom)}^2 < \infty\bigg\}
$$
the space of Hilbert--Schmidt operators from $\R^n$ to $L^2(\dom)$ endowed with
the norm
$$
  \|L\|_{\mathcal{L}_2(\R^n;L^2(\dom))}^2 := \sum_{k=1}^n\|L\tilde{\eta}_k\|_{L^2(\dom)}^2.
$$
The multiplicative noise term $\sigma: \Omega \times[0,T] \times L^2(\dom) \ni 
(\omega, t, u) \to\sigma(\omega, t, u) \in \R^{n\times n}$
with $\sigma=(\sigma_{ij})_{i,j=1,\ldots,n}$ 
is assumed to be $\mathcal{B}(L^2(\dom)\times[0,T])\otimes\mathcal{F};
\mathcal{B}(\mathcal{L}_2(\R^n;L^2(\dom)))$-measurable and $\mathbb{F}$-adapted.


\subsection{Assumptions and main result}\label{sec.main}

We impose the following assumptions.

\begin{labeling}{(A44)}
\item[(A1)] Domain: $\dom\subset\R^d$ ($d\le 3$) is a bounded domain with Lipschitz
boundary.

\item[(A2)] Initial datum: $u^0\in L^2(\Omega;L^\infty(\dom))$ is 
$\mathcal{F}_0$-measurable and $u_i^0(x)\in\DD$ for a.e.\ $x\in\dom$ $\Prob$-a.s.,
$i=1,\ldots,n$.

\item[(A3)] Diffusion matrix: $A=(A_{ij})\in C^0(\overline{\DD};\R^{n\times n})$ 
is Lipschitz continuous.

\item[(A4)] 
Multiplicative noise $\sigma:L^{2}(\mathcal{O})\to\mathcal{L}_2(\R^n;L^2(\dom))$
satisfies for some constant $C_\sigma>0$ and any $u\in L^2(\dom)$, $i,j,k=1,\ldots,n$,
$$
  \bigg\|\frac{\pa\sigma_{ij}}{\pa u^k}(u)\bigg\|_{\mathcal{L}(L^2(\dom);L^2(\dom))}
  \le C_\sigma.
$$

\item[(A5)] Entropy density: (i) There exists a convex function $h\in C^2(\DD;[0,\infty))
\cap C^0(\overline{\DD};[0,\infty))$ such that its derivative $h':\DD\to\R^n$ is 
invertible; (ii) there exist $c_{h}>0$,
$0 \le m < 1$ such that for all $u\in\DD$, $z\in\R^n$,
\begin{equation*}
  z^\top h''(u)A(u) z \ge c_{h}\sum_{i=1}^n \frac{z_i^2}{u_i^{2m}}.
\end{equation*}

\item[(A6)] Interaction of entropy density and noise: 
There exists $C_{h}>0$ such that for all $u\in\DD$,
\begin{align*}
  \max_{j=1,\ldots,n}\bigg|\sum_{i=1}^n\frac{\pa h}{\pa u_i}(u)\sigma_{ij}(u)\bigg|
	&+ \bigg|\sum_{i,j,k=1}^n\sigma_{kj}(u)
	\frac{\pa\sigma_{ij}}{\pa u_k}(u)\frac{\pa h}{\pa u_i}(u)\bigg| \\
	&{}+ \bigg|\sum_{i,j,k=1}^n\sigma_{ik}(u)\frac{\pa^2 h(u)}{\pa u_i\pa u_j}
	\sigma_{jk}(u)\bigg| \le C_h.
\end{align*}
\item[(A7)] Approximation of the entropy density:
Let 
$$ 
  [u_i]_\delta = \frac{u_i+\delta/n}{1+\delta}\quad\mbox{for }i=1,\ldots,n
$$
and set $[u]_\delta=([u_1]_\delta,\ldots,[u_n]_\delta)$
for $u\in\overline{\DD}$. It holds that for all $u\in\DD$ and $z\in\R^n$,
$$
  z^\top h''([u]_\delta)A(u)z - c_h\sum_{i=1}^{{n}}\frac{z_i^2}{[u_i]_\delta^{2m}}
	\ge z^\top R_\delta(u)z,
$$
where $R_\delta(u) \in \R^{n\times n}$ is a correction matrix that appears as a result 
of the compatibility of the regularized entropy $h([u]_\delta)$ with Assumption (A5ii), 
and {it holds that} $R_\delta(u)\to 0$ as $\delta\to 0$ uniformly in 
$\overline{\DD}$. 
\end{labeling}

\begin{remark}[Discussion of the assumptions]\rm
Assumptions (A1)--(A3), (A5) are essentially
the same conditions imposed in the (deterministic)
boundedness-by-entropy method \cite{Jue15}.
We assume additionally that the diffusion matrix is Lipschitz continuous, which
is needed to apply classical existence results for stochastic differential equations (see, e.g., \cite{PrRo07}). 
Assumption (A5ii) means that
the Onsager matrix is positive definite but not necessarily uniformly in $u$.
It provides gradient estimates for $u_i^{1-m}$, i.e., the diffusion matrix
has a fast-diffusion-type degeneracy. Assumption (A4) implies global Lipschitz
continuity for the multiplicative noise term, which is a standard condition
for SPDEs; see ,e.g.,\cite{PrRo07}. 
Assumption (A6) allows us to deal with the stochastic part when we derive the
entropy estimate. 
The idea is that the multiplicative noise is chosen in order to compensate possible
singularities of $h'(u)$ and $h''(u)$.
Finally, Assumption (A7) is needed since generally $h$ is not a $C^2(\overline{\DD})$
function and cannot be used in the It\^o lemma, whereas its regularization
$h_\delta(u)=h([u]_\delta)$ is a $C^2(\overline{\DD})$ function and therefore
admissible in the It\^o lemma. We suppose that $h_\delta$ is compatible with
Assumption (A5ii).
We present two examples from applications fulfilling Assumptions (A3)--(A7)
in Section \ref{sec.ex}.
\qed
\end{remark}

\begin{remark}[Extensions]\rm
{Our setting can be slightly generalized in different directions. 
The space dimension $d$ can be arbitrarily large. The condition
$d\le 3$ is needed to conclude the continuous embedding $H^3(\dom)\hookrightarrow
W^{1,\infty}(\dom)$. For general $d\ge 1$, we need to work with $H^s(\dom)$
with $s>1+d/2$ instead of $H^3(\dom)$.
We may include a nonlinear source term $F(u)$ (satisfying standard local Lipschitz 
continuity assumptions) which additionally interacts with the corresponding entropy
density \,\cite[Assumption H3, p.~86]{Jue16}, namely}
$$
  {\int_{\dom}F(u)\cdot h'(u)\,\textnormal{d}x
	\leq C_{F}\left(1+\int_{\dom}h(u)\,\textnormal{d}x\right).}
$$
{Moreover, we may allow for random initial data, i.e., we may prescribe an initial
probability measure instead of a given initial data. We refer to 
\cite[Remark 18]{DJZ19} for details. 
We consider only finite-dimensional Wiener processes instead of} {infinite-dimensional} 
{ones because we need to quantify the interaction of the entropy density and noise terms 
in Assumption (A6). Our technique also works with} {(trace-class) $Q$-Wiener processes} {
but the proof becomes very technical without introducing new ideas, 
which is the reason why we restrict ourselves to the finite-dimensional case.}
\qed
\end{remark}

Our main result is the global-in-time existence of martingale solutions to
\eqref{1.eq}--\eqref{1.bic}. First, we make precise the definition of martingale
solutions.

\begin{definition}[Global martingale solution]
For any fixed $T > 0$, the triple $(\widetilde U,\widetilde W,\widetilde u)$ is a 
{\em global martingale solution} to \eqref{1.eq}--\eqref{1.bic} if 
$\widetilde U=(\widetilde\Omega,\widetilde{\mathcal{F}},\widetilde{\Prob},
\widetilde{\mathbb{F}})$ is a stochastic basis with filtration 
$\widetilde{\mathbb{F}}=(\widetilde{\mathcal{F}}_t)_{t\in[0,T]}$, $\widetilde W$
is an $\R^n$-valued Wiener process {on this filtered probability space}, and $\widetilde u(t)
=(\widetilde u_1(t),\ldots,\widetilde u_n(t))$ is a progressively measurable
stochastic process for all $t\in[0,T]$ such that for all $i=1,\ldots,n$,
$$
  \widetilde u_i\in L^2(\widetilde\Omega;C^0([0,T];L^2_w(\dom)))\cap
	L^2(\widetilde\Omega;L^2(0,T;H^1(\dom))),
$$
the law of $\widetilde u_i(0)$ is the same as for $u_i^0$,
and $\widetilde u$ satisfies for all $\phi\in H^1(\dom)$ and $i=1,\ldots,n$,
\begin{align*}
  \int_\dom\widetilde u_i(t)\phi\,\,\textnormal{d}x
	&= \int_\dom\widetilde u_i(0)\phi\,\,\textnormal{d}x
	+ \sum_{j=1}^n\int_0^t\int_\dom A_{ij}(\widetilde u(s))\na \widetilde u_j(s)
	\cdot\na\phi\,\,\textnormal{d}x\,\textnormal{d}s \\
	&\phantom{xx}+ \sum_{j=1}^n\int_\dom\bigg(\int_0^t\sigma_{ij}(\widetilde u(s))
	\circ \textnormal{d}\widetilde W_j(s)\bigg)\phi\,\,\textnormal{d}x.
\end{align*}
\end{definition}

Here, $C^0([0,T];L_w^2(\dom))$ is the space of weakly continuous functions
$u:[0,T]\to L^2(\dom)$ such that $\sup_{0<t<T}\|u(t)\|_{L^2(\dom)}<\infty$.

\begin{theorem}[Existence of a global martingale solution]\label{thm.ex}
Let Assumptions (A1)--(A7) hold and let $T>0$. Then there exists a global martingale
solution to \eqref{1.eq}--\eqref{1.bic} satisfying 
$\widetilde u(x,t)\in\overline{\DD}$ for a.e.\ $(x,t)\in\dom\times(0,T)$
$\widetilde\Prob$-a.s.\ and $\widetilde u_i\in L^p(\widetilde\Omega;
L^\infty(0,T;L^\infty(\dom)))$ for any $p<\infty$.
\end{theorem}


\subsection{Key ideas}\label{sec.ideas}

We explain the strategy of the proof of Theorem \ref{thm.ex}. 
The approximation procedure combines the techniques of
\cite{DJZ19} and \cite{Jue15} and is illustrated in Figure \ref{fig}.

\begin{figure}[ht]
\begin{center}
\begin{tikzpicture}[scale=0.85]

\draw (-8.5,-5) circle[radius = 0.5cm] node {\footnotesize{$u$}};

\draw[->] (-8,-4.7) to node [pos=.5, above, sloped]{\footnotesize{Galerkin}} (-6,-3);
\draw[->] (-8,-5.3) to node [pos=.5, above, sloped]{\footnotesize{Galerkin}}(-6.1,-7);
\draw[->] (-8,-5.3) to node [pos=.5, below, sloped]{\footnotesize{Wong--Zakai}}(-6.1,-7);

\draw (-5.5,-3) circle[radius = 0.5cm] node {\footnotesize{$u^{(N)}$}};
\draw[->] (-5.5,-2.5) to node {} (-5.5, -2.1);
\node[align=center] at (-5.5,-1.5) {\footnotesize{SDEs solved up to} \\ 
\footnotesize{stopping time $\tau_R$}};

\draw[dashed,->](-5,-3) to node {} (-3,-4.7);

\draw (-5.5,-7) circle[radius = 0.6cm] node {\footnotesize{$u^{(N, \eta)}$}};
\draw[->] (-5.5,-7.6) to node {} (-5.5, -8.4);
\node[align = center] at (-5.5, -9) {\footnotesize{ODEs with random coefficients} \\
\footnotesize{solved up to $T>0$}};

\draw[->](-4.9,-7) to node [pos=.5, above, sloped]{\footnotesize{$\eta\rightarrow 0$}} 
(-3,-5.3);

\draw (-2.5,-5) circle[radius = 0.5cm] node {\footnotesize{$u^{(N)}$}};
\draw[->] (-2.5,-4.5) to node {}(-2.5, -4);
\node[align = center] at (-2.5, -3.3) {\footnotesize{entropy estimate} \\
\footnotesize{for $h_{\delta}(u)$}};
\draw[->] (-2.5, -5.5) to node {}(-2.5, -7);
\node[align =center] at (-2.5, -7.6) {\footnotesize{solution up to $T\wedge\tau_{R}$} \\
\footnotesize{$u^{N}(\omega,x,t)\in \overline{\DD}$}};

\draw[->] (-2,-5) to node [pos=.5, above]{\footnotesize{$\delta\rightarrow 0$}} (0,-5);

\draw (0.5,-5) circle[radius = 0.5cm] node {\footnotesize{$u^{(N)}$}};
\draw[->] (0.5,-4.5) to node {}(0.5, -3);
\node[align = center] at (0.5, -2.3) {\footnotesize{entropy estimate} \\ 
\footnotesize{for $h(u)$}};

\draw[->] (1,-5) to node [pos=.5, above]{\footnotesize{$R \to \infty$}} (3,-5);

\draw (3.5,-5) circle[radius = 0.5cm] node {\footnotesize{$u^{(N)}$}};
\draw[->] (3.5, -5.5) to node {}(3.5, -7);
\node[align =center] at (3.5, -7.6) {\footnotesize{global solution} \\
\footnotesize{$u^{(N)}(\omega,x,t) \in \overline{\DD}$}};

\draw[->] (4,-5) to node [pos=.5, above]{\footnotesize{$N \to \infty$}} (6,-5);

\draw (6.5,-5) circle[radius = 0.5cm] node {\footnotesize{$\widetilde{u}$}};
\draw[->] (6.5, -4.5) to node {}(6.5, -3);
\node[align =center] at (6.5, -2.3) {\footnotesize{global martingale solution} \\
\footnotesize{$\widetilde{u}(\omega,x,t) \in \overline{\DD}$}};
\end{tikzpicture}
\end{center}
\caption{Steps of the proof of the existence theorem.}
\label{fig}
\end{figure}
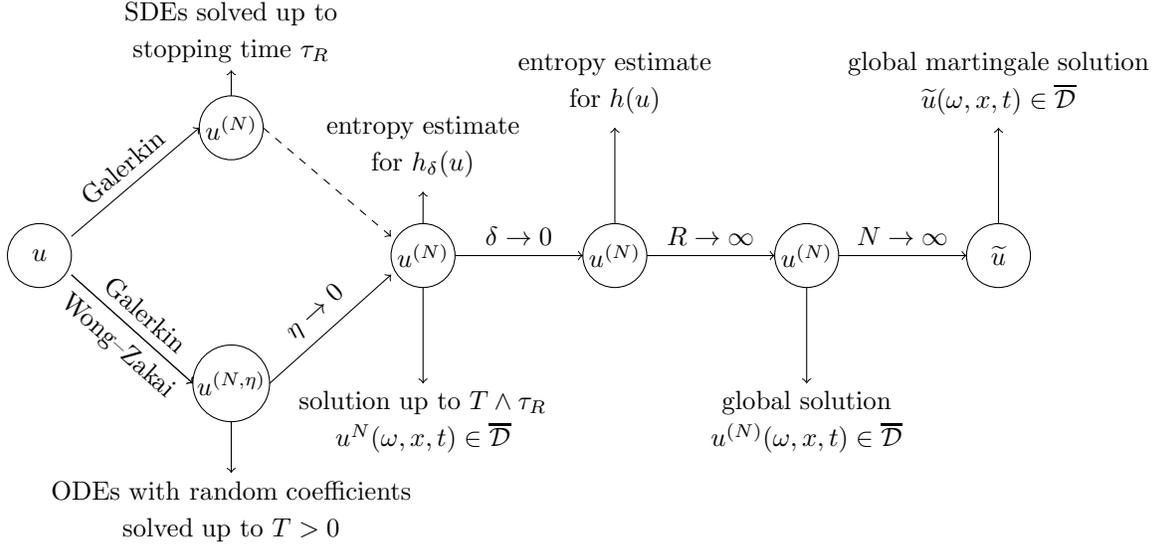

{\em Step 1: Stochastic Galerkin approximation.} Equations \eqref{1.eq} are
projected on a Galerkin space with finite dimension $N\in\N$. The 
existence of a unique solution $u^{(N)}$ to the stochastic differential system 
up to a stopping time $\tau_R$ is shown by Banach's fixed-point theorem, 
exploiting the Lipschitz continuity of the nonlinearities. 
{We recall that $R > 0$ is a previously chosen parameter in the definition 
of {\em the} stopping time $\tau_R$, describing the upper bound of a certain norm.}
Since the contraction constant depends on $R$, 
we cannot pass to the limit $R\to \infty$.
For global solutions, we need a priori estimates which can be derived in principle
from the entropy inequality, similar to \eqref{1.ei}. However, this requires
that the solution is positive and bounded, which cannot be deduced from this
technique. We need the boundedness-by-entropy method.

{\em Step 2: Wong--Zakai approximation.} In order to obtain the uniform boundedness
for the solutions, we regularize the noise in the sense of the Wong--Zakai approximation
with parameter $\eta>0$, giving a system of ordinary differential
equations, which is parametrized by the stochastic variable. 
The existence of a solution $u^{(N,\eta)}$ follows from the 
boundedness-by-entropy method \cite{Jue15}. A consequence of this technique is the
nonnegativity and boundedness of $u^{(N,\eta)}(x,t)$ $\Prob$-a.s. 
We also obtain estimates from an entropy inequality, but they depend on $\eta$
and therefore cannot be further applied. 
The Wong--Zakai theory allows us to pass to the limit $\eta\to 0$, showing that
$u^{(N,\eta)}$ converges to the solution $u^{(N)}$ obtained in Step 1. 
Since this solution is unique, we deduce that
$u^{(N)}$ is nonnegative and bounded, more precisely
$u^{(N)}(x,t)\in\overline{\DD}$ for a.e.\ $(x,t)\in\dom\times(0,T)$ $\Prob$-a.s.

{\em Step 3: Entropy estimates.} Gradient estimates uniform in $N$ are obtained from 
the entropy inequality, which is derived in the deterministic setting
by using the test function $h'(u^{(N)})$. Since the entropy density $h$ generally
does not belong to $C^2(\overline{\DD})$, we cannot use the It\^o lemma. We need to
regularize the entropy density by a function $h_\delta$ (with parameter $\delta>0$)
which belongs to $C^2(\overline{\DD})$. It\^o's lemma then allows us to derive
entropy estimates which are uniform in $\delta$, $R$, and $N$.
After passing to the limit $\delta\to 0$, we infer the following
entropy estimates uniform in the Galerkin dimension $N$:
\begin{equation}\label{1.sei}
  \E\int_\dom h(u^{{(N)}}(t))\,\,\textnormal{d}x + C_1\E\int_0^t\int_\dom\sum_{i=1}^n
	\big|\na (u_i^{(N)})^{1-m}\big|^2 \,\,\textnormal{d}x\,\textnormal{d}s \le C_2,
\end{equation}
where $C_1$, $C_2>0$ are independent of $N$ and $R$ and $m<1$. Since the right-hand side
does not depend on $R$, we may pass to the limit $R\to\infty$, thus obtaining
global approximate solutions $u^{(N)}$.

{\em Step 4: Tightness of the laws.} The tightness of the laws of $(u^{(N)})$
in a sub-Polish space is shown by applying the tightness criterion 
of Brze\'{z}niak and Motyl \cite{BrMo14}. It involves the verification of
some a priori estimates which can be deduced from \eqref{1.sei}.

{\em Step 5: Convergence.} The tightness of the laws of $(u^{(N)})$ and the
Skorokhod--Jakubowski theorem allow us to perform the limit $N\to\infty$
in the sense that there exist random variables $\widetilde u^{(N)}$,
with the same law as $u^{(N)}$, converging to a martingale solution to
\eqref{1.eq}--\eqref{1.bic}. Unfortunately, the property 
$u^{(N)}(x,t)\in\overline{\DD}$ does not directly imply 
that $\widetilde u^{(N)}(x,t)\in\overline{\DD}$ since only the laws of
these random variables coincide. Our idea is to show, using the Kuratowski theorem, 
that $\sum_{i=1}^n\|\widetilde u_i^{(N)}\|_{L^\infty}\le 1$ $\widetilde\Prob$-a.s.\
and that $\widetilde u_i^{(N)}$ lies in the union of the unit balls around zero 
and around one (with respect to the $L^\infty$ norm)
from which we conclude that $\widetilde u(x,t)\in\overline{\DD}$
$\widetilde\Prob$-a.s. 

These steps are detailed in Section \ref{sec.proof}. 
Two examples fulfilling Assumptions (A3)--(A7) are presented in Section \ref{sec.ex}
and some theorems from stochastic analysis are recalled in Appendix \ref{sec.tech}.


\section{Existence analysis}\label{sec.proof}

We prove Theorem \ref{thm.ex} by approximating system \eqref{1.eq} by a stochastic
Galerkin method and later by a Wong--Zakai type approximation of the 
$\mathbb{R}^{n}$-valued Wiener process.

\subsection{Stochastic Galerkin approximation}\label{sec.sga}

We prove the existence of a strong (in the probability sense) solution to
an approximate system up to a stopping time by using the Banach fixed-point
theorem.

The approximate system is obtained from projecting \eqref{1.eq} onto the
finite-dimensional Hilbert space $H_N=\operatorname{span}\{e_1,\ldots,e_N\}$,
where $N\in\N$ and $(e_j)_{j\in\N}$ is an orthonormal basis of $L^2(\dom)$ {such that $H_N \subset H^1(\dom)\cap L^\infty(\dom)$}.
We introduce the projection operator $\Pi_N:L^2(\dom)\to H_N$ by
$$
  \Pi_N(v) = \sum_{i=1}^N \langle v,e_i\rangle e_i\quad\mbox{for }v\in L^2(\dom).
$$
We need the basis in $H^1(\dom)\cap L^\infty(\dom)$ for later purposes, i.e in the proof of Proposition~\ref{prop.Neta}.

The approximate problem is the following system of stochastic differential
equations,
\begin{align}
  \textnormal{d}u_i^{(N)} &= \Pi_N\diver\bigg(\sum_{j=1}^n A_{ij}(u^{(N)})\na u_j^{(N)}
	\bigg) \,\textnormal{d}t	+ \sum_{j=1}^n\Pi_N\big(\sigma_{ij}(u^{(N)})\big)\,\textnormal{d}W_j(t)
	 \nonumber \\
	&\phantom{xx}{}+ \frac12 \Pi_N \bigg(\sum_{k=1}^n\sum_{j=1}^n 
	\sigma_{kj}(u^{(N)})\frac{\pa\sigma_{ij}}{\pa u_k}(u^{(N)})\bigg)\,\textnormal{d}t,
	\quad i=1,\ldots,N, \label{2.approx1}
\end{align}
with the initial conditions
\begin{equation}\label{2.approx2}
	u_i^{(N)}(0) = \Pi_N(u_i^0), \quad i=1,\ldots,N. 
\end{equation}
Since the solutions $u^{(N)}$ may not lie in the Gibbs simplex $\overline{\DD}$,
we need to extend the functions $A_{ij}$ and $\sigma_{ij}$ to the whole space
$\R^n$. This is done in such a way that $A_{ij}$ and $\sigma_{ij}$ are Lipschitz
continuous on $\R^n$ (we do not change the notation). This implies that
$A_{ij}$ and $\sigma_{ij}$ grow at most linearly.

Given $T>0$, we introduce the space $X_T=L^2(\Omega; {C}([0,T];H_N))$ with the norm
$\|u\|_{X_T}^2:=\E(\sup_{0\leq t\leq T}\|u(t)\|_{H_N})^2$. For given $R>0$ and $u\in X_T$,
we define the stopping time
$$
  \tau_R:=\inf\{t\in[0,T]:\|u(t)\|_{H^1(\dom)}>R\}.
$$
Furthermore, we introduce the It\^o correction operator
$\mathcal{T}=(\mathcal{T}_1,\ldots,\mathcal{T}_n):L^2(\dom;\R^n)\to L^2(\dom;\R^n)$ by
$$
  \mathcal{T}_i(u) = \sum_{k=1}^n\sum_{j=1}^n \sigma_{kj}(u)
	\frac{\pa\sigma_{ij}}{\pa u_k}(u), \quad u\in L^2(\dom;\R^n).
$$

\begin{proposition}\label{prop.ex}
Let $T>0$, $R>0$ {be fixed}, and let Assumptions (A1)--(A5) hold.
 Then there exists a unique strong (in the probabilistic sense) solution $u^{(N)}\in
X_{T\wedge\tau_R}$ to \eqref{2.approx1}--\eqref{2.approx2}
such that  for any $t\in[0,T\wedge\tau_R]$,
\begin{align}
  \langle u^{(N)}(t), \phi\rangle 
	&= \langle u^0,\phi\rangle - \int_0^t \langle A(u^{(N)}(s))\na u^{(N)}(s), 
	\na \phi \rangle \,\textnormal{d}s \nonumber \\
  &\phantom{xx}{}+ \frac12 \int_0^t \langle \mathcal{T}(u^{(N)}(s)), \phi \rangle \,\textnormal{d}s 
	+ \int_0^t \langle \sigma(u^{(N)}(s))\,\textnormal{d}W(s),\phi\rangle  \label{2.uNweak}
\end{align}
for any $\phi=(\phi_1,\ldots,\phi_N)\in {W^{1,\infty}(\dom; \R^n)\cap H_N^n}$ {and $X_{T\wedge\tau_R} := L^2(\Omega; C([0,T\wedge\tau_R]; H_N))$.}
\end{proposition}

\begin{proof}
The idea of the proof is to apply the Banach fixed-point theorem to the mapping
$S:X_{T}\to X_{T}$,
\begin{align}
  \langle S(u^{(N)})(t),\phi\rangle &= \langle u^0,\phi\rangle
	- \int_0^t\langle A(u^{(N)})(s)\na u^{(N)}(s),\na\phi\rangle \,\textnormal{d}s \nonumber \\
	&\phantom{xx}{}+ \frac12\int_0^t\langle\mathcal{T}(u^{(N)}(s)),\phi\rangle \,\textnormal{d}s
	+ \int_0^t\langle\sigma(u^{(N)}(s))\,\textnormal{d}W(s),\phi\rangle , \label{2.S}
\end{align}
where $u^{(N)}\in X_T$ and $\phi\in {W^{1,\infty}(\dom; \R^n)\cap H_N^n}$.
The linear growth of $A$ and $\sigma$ allows us to show that 
$S$ indeed maps $X_{T}$ into itself and that $S$ is a contraction for some
$T^*\in(0,T\wedge\tau_R]$. Although the arguments are rather standard, we provide
a full proof for completeness.

We show first the self-mapping property. Let $u\in X_T$ and $\phi\in 
{W^{1,\infty}(\dom; \R^n)\cap H_N^n}$. Then Definition \eqref{2.S} gives
\begin{align*}
  \|\langle & S(u),\phi\rangle\|^2_{L^2(\Omega;C[0,T\wedge\tau_R])} 
	= \E\bigg(\sup_{0\leq t\leq T\wedge\tau_R}|\langle S(u)(t),\phi\rangle|\bigg)^2 \\
	&\le \|\phi\|_{L^2(\dom)}^2\E\|u^0\|_{L^2(\dom)}^2
	+ C\E\int_0^{T\wedge\tau_R}|\langle A(u(s))\na u(s),\na\phi\rangle|^2 \,\textnormal{d}s \\
	&\phantom{xx}{}
	+ C\E\int_0^{T\wedge\tau_R}|\langle\mathcal{T}(u(s)),\phi\rangle|^2 \,\textnormal{d}s
	+ C\E\bigg(\sup_{0\leq t\leq T\wedge\tau_R}\bigg|\int_0^t\langle\sigma(u(s))\,\textnormal{d}W(s),\phi\rangle
	\bigg|\bigg)^2 \\
	&=: I_1+\cdots+I_4.
\end{align*}
We estimate the terms $I_2$, $I_3$, and $I_4$. Because of the linear growth
of $A$ and the equivalence of the norms in $H_N$, we find that
\begin{align*}
  I_2 &\le C\|\na\phi\|_{L^\infty(\dom)}^2\E\int_0^{T\wedge\tau_R}
	(1+\|u(s)\|_{L^2(\dom)}^2)\|\na u(s)\|_{L^2(\dom)}^2 \,\textnormal{d}s \\
	&\le C(T\wedge\tau_R)\|\na\phi\|_{L^\infty(\dom)}^2
	\E\bigg(1+\sup_{0\leq t\leq T\wedge\tau_R}\|u(t)\|_{L^2(\dom)}^2\bigg)R^2 \\
	&\le C(N,R)T\|\phi\|_{H_N}^2\big(1+\|u\|_{X_{T\wedge\tau_R}}^2\big).
\end{align*}
Assumption (A4) implies that $\mathcal{T}(u)$ grows at most linearly, so
\begin{align*}
  I_3 &\le C\|\phi\|_{L^2(\dom)}^2\E\int_0^{T\wedge\tau_R}
	\|\mathcal{T}(u(s))\|_{L^2(\dom)}^2 \,\textnormal{d}s \\
	&\le C\|\phi\|_{L^2(\dom)}^2\E\int_0^{T\wedge\tau_R}
	\big(1+\|u(s)\|_{L^2(\dom)}^2\big) \,\textnormal{d}s
	\le C(N)T\|\phi\|_{H_N}^2\big(1+\|u\|_{X_{T\wedge\tau_R}}^2\big).
\end{align*}
We obtain from the Burkholder--Davis--Gundy inequality \cite[Prop. 2.12]{Kru14} 
\begin{align*}
  I_4 &\le C\|\phi\|_{L^2(\dom)}^2\E\int_0^{T\wedge\tau_R}
	\|\sigma(u(s))\|_{\mathcal{L}_2(\R^n;L^2(\dom))}^2 \,\textnormal{d}s \\
	&\le C\|\phi\|_{L^2(\dom)}^2\E\int_0^{T\wedge\tau_R}
	\big(1+\|u(s)\|_{L^2(\dom)}^2\big) \,\textnormal{d}s
	\le C(N)T\|\phi\|_{L^2(\dom)}^2\big(1+\|u\|_{X_{T\wedge\tau_R}}^2\big).
\end{align*}
Summarizing these estimates, we find that
$$
  \|S(u)\|_{X_{T\wedge\tau_R}}^2 \le C\E\|u^0\|_{L^2(\dom)}^2 
	+ C(N,R)T\big(1+\|u\|_{X_{T\wedge\tau_R}}^2\big),
$$
which implies that $S$ maps $X_{T\wedge\tau_R}$ to $X_{T\wedge\tau_R}$.

Next, we show that $S:X_T\to X_T$ is a contraction if $0<T<\tau_R$ is sufficiently
small. Let $u$, $v\in X_T$, $\phi\in {W^{1,\infty}(\dom; \R^n)\cap H_N^n}$, 
and $R>0$ and set
$$
  \tau_R = \inf\big\{t\in[0,T]:\|u(t)\|_{H^1(\dom)}>R\big\}\wedge
	\inf\big\{t\in[0,T]:\|v(t)\|_{H^1(\dom)}>R\big\}.
$$
Then
\begin{align*}
  \|\langle & S(u)-S(v),\phi\rangle\|_{L^2(\Omega;C[0,T\wedge\tau_R])}^2 \\
	&\le {C} \E\bigg(\sup_{0\leq t\leq T\wedge\tau_R}\bigg|\int_0^t\langle A(u(s))\na u(s)
	- A(v(s))\na v(s),\na\phi(s)\rangle \,\textnormal{d}s\bigg|\bigg)^2 \\
	&\phantom{xx}{}+ {\frac{C}{2}}\E\bigg(\sup_{0\leq t\leq T\wedge\tau_R}\bigg|\int_0^t\langle
	\mathcal{T}(u(s))-\mathcal{T}(v(s)),\phi\rangle \,\textnormal{d}s\bigg|\bigg)^2 \\
	&\phantom{xx}{}+ {C}\E\bigg(\sup_{0\leq t\leq T\wedge\tau_R}\bigg|\int_0^t\big\langle
	\big(\sigma(u(s))-\sigma(v(s))\big)\,\textnormal{d}W(s),\phi\big\rangle\bigg|\bigg)^2 \\
	&=: I_5+I_6+I_7.
\end{align*}
Assumption (A3) shows that
\begin{align*}
  I_5 &\le C {T}\E\int_0^{T\wedge\tau_R}\big|\big\langle (A(u)-A(v))\na u
	+ A(v)\na(u-v),\na\phi\big\rangle\big|^2 \,\textnormal{d}s \\
	&\le C{T}\|\na\phi\|_{L^\infty(\dom)}^2\E\int_0^{T\wedge\tau_R}
	\Big(\|u(s)-v(s)\|_{L^2(\dom)}^2\|\na u(s)\|_{L^2(\dom)}^2 \\
	&\phantom{xx}{}+ \big(1+\|v(s)\|_{L^2(\dom)}^2\big)
	\|\na(u-v)(s)\|_{L^2(\dom)}^2\Big)\,\textnormal{d}s \\
	&\le C(N)R^2{T^2}\|\na\phi\|_{L^\infty(\dom)}^2\|u-v\|_{X_{T\wedge\tau_R}}^2.
\end{align*}
Similarly, exploiting the linear growth of $\sigma$ and $\mathcal{T}$,
\begin{align*}
  I_6 &\le C{T}\E\int_0^{T\wedge\tau_R}|\langle\mathcal{T}(u)-\mathcal{T}(v),
	\phi\rangle|^2 \,\textnormal{d}s \le C{T^2}\|\phi\|_{L^2(\dom)}^2\|u-v\|_{X_{T\wedge\tau_R}}^2, \\
	I_7 &\le C{T}\|\phi\|_{L^2(\dom)}^2\E\int_0^{T\wedge\tau_R}\|\sigma(u)-\sigma(v)
	\|^2_{\mathcal{L}_2(\R^n;L^2(\dom))} \,\textnormal{d}s
	\le C{T^2}\|\phi\|_{L^2(\dom)}^2\|u-v\|_{X_{T\wedge\tau_R}}^2.
\end{align*}
Consequently,
$$
  \|S(u)-S(v)\|_{X_{T\wedge\tau_R}} \le C(N,R){T^2}\|u-v\|_{X_{T\wedge\tau_R}},
$$
which shows that $S:X_{T^*}\to X_{T^*}$ is a contraction for $0<T^*<T\wedge\tau_R$
satisfying $C(N,R){(T^*)^2}<1$.

By the Banach fixed-point theorem, there exists a unique
fixed point $u^{(N)}\in X_{T^*}$, which means that $u^{(N)}$ solves
\eqref{2.uNweak} for any $t\in(0,T^*)$. 
The local solution can be uniquely extended to a global
one on the whole interval $[0,T\wedge\tau_R]$ since $T^*>0$ is independent of the 
initial datum. Standard results \cite[Lemma 3.23]{KuNe18} show that the
stopping time $\tau_R$ is $\Prob$-a.s.\ positive. 
\end{proof}


\subsection{Wong--Zakai-type approximation}\label{sec.wzga}

We prove the existence of global-in-time solutions to another approximate system
of \eqref{1.eq}, consisting of a system of ordinary differential equations (ODE).
For this, we introduce
two levels of approximations with the following parameters: the Galerkin dimension 
$N\in\N$ and a Wong--Zakai type approximation of the $\mathbb{R}^n$-valued
Wiener process with time step $\eta>0$. More precisely, we project
\eqref{1.eq} as in the previous subsection onto the finite-dimensional Galerkin space
$H_N$ and introduce a uniform partition of the time interval $[0,T]$
with time step $\eta=T/M$, where $M\in\N$ and $t_k=k\eta$ for $k=0,\ldots,M$.
The Wiener process is approximated by the process
\begin{equation}\label{2.wong}
  W^{(\eta)}(t) = W(t_k) + \frac{t-t_k}{\eta}(W(t_{k+1})-W(t_k)), \quad
	t\in[t_k,t_{k+1}],\ k=0,\ldots,M.
\end{equation}

Approximations like this or via convolution with a smooth kernel are generally referred 
to as Wong--Zakai approximations and were introduced in \cite{WoZa65} in one dimension 
and in \cite{StVa72} for systems. Further generalizations can be found in 
\cite{MaZhu19}, \cite{Twa92}--\cite{Twa96}.

The approximate equations read as
\begin{equation}\label{2.approx}
  \frac{\textnormal{d}u^{(N,\eta)}}{\textnormal{d}t} = \Pi_N\diver\bigg(A(u^{(N,\eta)})
	\na u^{(N,\eta)}\bigg) + \Pi_N\big(\sigma(u^{(N,\eta)})\big)
	\frac{\textnormal{d}W^{(\eta)}}{\textnormal{d}t},
\end{equation}
with the initial conditions 
\begin{equation}\label{2.ic}
  u^{(N,\eta)}(0)=\Pi_N(u^0).
\end{equation}
This is a finite-dimensional system of ODEs. 
The existence of global-in-time solutions is deduced from the boundedness-by-entropy
technique of \cite{Jue15}.

\begin{proposition}\label{prop.Neta}
Let $T>0$, $N\in\N$, $\eta>0$, and let Assumptions (A1)--(A5) hold. 
Then for almost all $\omega\in\Omega$, there exists
a global-in-time weak solution $u^{(N,\eta)}=(u_1^{(N,\eta)},\ldots,u_n^{(N,\eta)})$ 
to \eqref{2.approx}--\eqref{2.ic} satisfying
$$
  u_i^{(N,\eta)}(\omega,\cdot,\cdot)\in L^2(0,T;H^1(\dom)),
	\quad \pa_t u_i^{(N,\eta)}(\omega,\cdot,\cdot)\in L^2(0,T;H^1(\dom)')
$$
for $i=1,\ldots,n$ and a.e.\ $\omega\in\Omega$, 
$$
  u^{(N,\eta)}(x,t)\in\overline{\DD}\quad\mbox{for }(x,t)\in\dom\times(0,T)
	\ \Prob\mbox{-a.s.},
$$
$u^{(N,\eta)}(0)=\Pi_N(u^0)$ in the sense of $H^1(\dom)'$, and
\begin{align*}
  \langle u^{(N,\eta)}(t),\phi\rangle 
	&= \langle u^0,\phi\rangle - \int_0^t\langle A(u^{(N,\eta)}(s))\na u^{(N,\eta)}(s),
	\na\phi\rangle \,\textnormal{d}s \\
	&\phantom{xx}{}+ \int_0^t\bigg\langle\sigma(u^{(N,\eta)}(s))
	\frac{\textnormal{d}W^{(\eta)}}{\textnormal{d}t}(s),\phi\bigg\rangle \,\textnormal{d}s
\end{align*}
for any $\phi\in L^2(0,T;H^1(\dom)\cap H_N)^n$.
\end{proposition}

\begin{proof}
In principle, the proof follows by applying the boundedness-by-entropy
method \cite[Theorem 2]{Jue15} to the cross-diffusion system
\eqref{2.approx} with the source term 
\begin{equation*}
  f(u,t):=\Pi_N(\sigma(u^{(N,\eta)}(t)))\frac{\textnormal{d}W^{(\eta)}}{\textnormal{d}t}(t).
\end{equation*}
{We drop the $\omega$ dependence to simplify
the notation.} 
For the convenience of those readers who are not familiar with this technique,
we recall the main steps of the proof. Details can be found in \cite{Jue15,Jue16}.

The idea is to {formulate} \eqref{2.approx} 
as a finite-dimensional diffusion problem 
with variable $w=h'(u^{(N,\eta)})$. After solving this problem in $w$,
we can then define $u^{(N,\eta)}:=(h')^{-1}(w)$, and since the range of $(h')^{-1}$
is the bounded set $\DD$, we find that $u^{(N,\eta)}(\omega,x,t)\in\DD$ for a.e.\ 
$\omega\in\Omega$. The transformation causes two difficulties: First,
the flux transforms to $A(u^{(N,\eta)})\na u^{(N,\eta)}
= B(w)\na w$, but the new diffusion matrix $B(w)=A(u^{(N,\eta)})h''(u^{(N,\eta)})^{-1}$ 
is generally only positive semi-definite. Second,
the time derivative becomes $\pa_t u^{(N,\eta)}
=h''(u^{(N,\eta)})\pa_t w$, but $h''(u^{(N,\eta)})$ may be not invertible on $\pa\DD$.
Both issues can be solved by discretizing \eqref{2.approx} in time and
adding a regularization. In fact, for the fixed $T > 0$, $L \in \N$, we consider a 
time grid $\pi_L$ (which is finer than the uniform time partition considered for the
Wong--Zakai type approximation), and set $\tau=T/L>0$. Let 
$\eps>0$ and $w^{k-1}\in L^\infty(\dom;\R^n)$ be given. We wish to solve {i.e.~find $w^{k}\in H^{1}(\mathcal{O};\mathbb{R}^{n})$, such that}
\begin{align}
  \frac{1}{\tau}\int_\dom\big(u(w^k)-u(w^{k-1})\big)\cdot\phi \,\textnormal{d}x
	&+ \int_\dom\sum_{i,j=1}^n B_{ij}(w^k)\na\phi_i\cdot\na w_j^k \,\textnormal{d}x \nonumber \\
	&{}+ \eps\int_\dom w^k\cdot\phi \,\textnormal{d}x
	= \int_\dom f(u(w^k),t_{k})\cdot\phi \,\textnormal{d}x, \label{2.approxw}
\end{align}
where $u(w):=(h')^{-1}(w)$ and $\phi\in H^1(\dom;\R^n)$. 

{\em Step 1: Solution of the approximate problem.} 
We prove the existence of a solution to \eqref{2.ic} and \eqref{2.approxw}
by applying the Leray--Schauder fixed-point theorem. 
Let the Galerkin space $H_N$ be a subset of $H^1(\dom;\R^n)$ such that
$H_N\subset L^\infty(\dom;\R^n)$. 
(This is possible by choosing appropriate basis functions.)
Let $y\in L^\infty(\dom;\R^n)$ and $\vartheta\in[0,1]$ be given. 
We consider the following linear problem: Find $w=w^k\in H_N$ such that
\begin{equation}\label{2.LM}
  a(w,\phi)=F(\phi)\quad\mbox{for all }\phi\in H_N,
\end{equation}
where
\begin{align*}
  a(w,\phi) &= \int_\dom\sum_{i,j=1}^nB_{ij}(y)\na\phi_i\cdot\na w_j \,\textnormal{d}x
	+ \eps\int_\dom w\cdot\phi \,\textnormal{d}x, \\
  F(\phi) &= -\frac{\vartheta}{\tau}\int_\dom\big(u(y)-u(w^{k-1})\big)\cdot\phi \,\textnormal{d}x
	+ \vartheta\int_\dom f(u(y),t_{k})\cdot\phi \,\textnormal{d}x.
\end{align*}
The boundedness of $y$ and the Cauchy--Schwarz inequality show that $a$ and $F$
are bounded on $H_N$. Since $B(y)$ is positive semi-definite and 
all norms are equivalent in finite dimensions,
$$
  a(w,w) \ge \eps\|w\|_{L^2(\dom)}^2 \ge \eps C(N)\|w\|_{H^1(\dom)}^2,
$$
which means that $a$ is coercive on $H_N$. By the Lax--Milgram lemma,
there exists a unique solution $w\in H_N$ to \eqref{2.LM}
and it holds that $w\in L^\infty(\dom;\R^n)$. This defines the fixed-point
operator $S:H_N\times[0,1]\to H_N$,
$S(y,\vartheta)=w$, where $w$ solves \eqref{2.LM}. 

We verify the assumptions of the Leray--Schauder theorem. The only solution
to \eqref{2.LM} with $\vartheta=0$ is $w=0$; thus $S(y,0)=0$. The continuity of $S$
follows from standard arguments; see the proof of \cite[Lemma 5]{Jue15} for details.
Since $H_N$ is finite-dimensional, $S$ is compact. It remains to prove a uniform
bound for all fixed points of $S(\cdot,\vartheta)$. Let $w\in H_N$ be such a fixed point.
Then $w$ solves \eqref{2.LM} with $y$ replaced by $w$. Choosing the test function
$\phi=w$, we obtain $\Prob$-a.s.
\begin{align}
  \frac{\vartheta}{\tau}\int_\dom\big(u(w)-u(w^{k-1})\big)\cdot w \,\textnormal{d}x
	&+ \int_\dom\sum_{i,j=1}^n B_{ij}(w)\na w_i\cdot\na w_j \,\textnormal{d}x \nonumber \\
	&{}+\eps\int_\dom|w|^2 \,\textnormal{d}x = \vartheta\int_\dom f(u(w),t_{k})\cdot w\,\textnormal{d}x. \label{2.ei}
\end{align}
The convexity of $h$ (see Assumption (A5i)) shows that
$$
  \frac{\vartheta}{\tau}\int_\dom\big(u(w)-u(w^{k-1})\big)\cdot w \,\textnormal{d}x
	\ge \frac{\vartheta}{\tau}\int_\dom\big(h(u(w))-h(u(w^{k-1}))\big)\,\textnormal{d}x.
$$
Since $B(w)$ is positive semi-definite, we have 
$\sum_{i,j=1}^n B_{ij}(w)\na w_i\cdot\na w_j \ge 0$.
Finally, we use Assumption (A4), \eqref{2.wong} along with Kolmogorov's continuity theorem to infer that for all $u\in[0,\infty)^n$, 
\begin{align*}
  f(u(w),t_{k})\cdot h'(u(w)) &= \frac{1}{\eta}\sum_{i,j=1}^n\sigma_{ij}(u(w))
	(W_j(t_{k+1})-W_j(t_k))\frac{\pa h}{\pa u_i}(u(w)) \\
  &\le\frac{1}{\eta}\sum_{j=1}^n|W_j(t_{k+1})-W_j(t_k)|
	\max_{j=1,\ldots,n}\sum_{i=1}^n\bigg|\sigma_{ij}(u(w))\frac{\pa h}{\pa u_i}(u(w))\bigg|
	\le C(\eta).
\end{align*}
This shows that the right-hand side of \eqref{2.ei} is bounded uniformly
in $\vartheta$ and $w$. We infer that 
$\eps\|w\|_{L^2(\dom)}^2 \le C(\eta)$ and consequently 
$\|w\|_{H^1(\dom)}\le C(\eta,\eps,N)$ $\Prob$-a.s.
This yields the desired uniform bound,
and we can apply the Leray--Schauder fixed-point theorem to conclude the
existence of a weak solution $w^k\in H_N$ to \eqref{2.approxw}. 

{\em Step 2: Uniform estimates.}
Since we do not have any uniform estimates for $w$, we switch to the original
variable $u(w^k)$. Let $w^{(\tau)}(\omega,x,t)=w^k(\omega,x)$ and
$u^{(\tau)}(\omega,x,t)=u(w^k(\omega,x))$ for $\omega\in\Omega$, $x\in\dom$,
and $t\in((k-1)\tau,k\tau]$, $k=1,\ldots,L$. At time $t=0$, we set
$w^{(\tau)}(\cdot,0)=h'(u^0)$ and $u^{(\tau)}(\cdot,0)=u^0$. We also need the
shift operator
$(\Gamma_\tau u^{(\tau)})(\omega,x,t)=u(w^{k-1}(\omega,x))$ for 
$\omega\in\Omega$, $x\in\dom$, and $t\in((k-1)\tau,k\tau]$. In this notation,
the weak formulation \eqref{2.approxw} can be written as
\begin{align}
  \frac{1}{\tau}\int_0^T\int_\dom(u^{(\tau)}-\Gamma_\tau u^{(\tau)})\cdot\phi \,\,\textnormal{d}x\,\textnormal{d}t
	&+ \int_0^T\int_\dom\sum_{i,j=1}^n A_{ij}(u^{(\tau)})
	\na\phi_i\cdot\na u_j^{(\tau)} \,\,\textnormal{d}x\,\textnormal{d}t 
	\nonumber \\
	&{}+ \eps\int_0^T\int_\dom w^{(\tau)}\cdot\phi \,\,\textnormal{d}x\,\textnormal{d}t
	= \int_0^T\int_\dom f(u^{(\tau)})\cdot\phi \,\,\textnormal{d}x\,\textnormal{d}t \label{2.weaku}
\end{align}
for piecewise constant functions $\phi:(0,T)\to H_N$. 

We derive now some uniform estimates, {using the test function
$\phi=w^{(\tau)}$ in \eqref{2.weaku}}. 
At this point, we need Assumption (A5ii):
\begin{align*}
  {\sum_{i,j=1}^n A_{ij}(u^{(\tau)})\na w_i^{(\tau)}\cdot\na u_j^{(\tau)}}
  &= \sum_{i,j=1}^n {\big(h''(u^{(\tau)})A(u^{(\tau)})\big)_{ij} 
	\na u_i^{(\tau)}}\cdot\na u_j^{(\tau)} \\
	&\ge c_h\sum_{i=1}^n \frac{|\na u^{(\tau)}_i|^2}{(u^{(\tau)}_i)^{2m}} 
	= \frac{c_h}{(1-m)^2}\sum_{i=1}^n|\na (u^{(\tau)}_i)^{1-m}|^2.
\end{align*}
Hence, summing \eqref{2.weaku} over {$k=1,\ldots,\ell$ with $\ell\le L$}, it follows 
similarly as in Step 1 that $\Prob$-a.s.
$$
  \int_\dom h(u(w^{{\ell}}))\,\textnormal{d}x + \tau\sum_{k=1}^{{\ell}}\sum_{i=1}^n\int_\dom
	|\na u_i(w^k)^{1-m}|^2 \,\textnormal{d}x + \eps\tau\sum_{k=1}^{{\ell}}\|w^k\|_{{L^2(\dom)}}^2 \le C,
$$
where $C>0$ depends on $h(u^0), \eta$ but not on $\eps$ or $\tau$. 
Together with the uniform $L^\infty$ bound for $u^{(\tau)}$, this yields
$$
  \|(u^{(\tau)})^{1-m}\|_{L^2(0,T;H^1(\dom))} 
	+ \sqrt{\eps}\|w^{(\tau)}\|_{L^2(0,T;{L^2(\dom)})} \le C.
$$
Moreover, $\na u^{(\tau)}=(1-m)^{-1}(u^{(\tau)})^m\na (u^{(\tau)})^{1-m}$
is uniformly bounded in $L^2(\dom\times(0,T))$. 
(Here, we need that $0\le m<1$.) A straightforward computation shows
that $\tau^{-1}(u^{(\tau)}-\Gamma_\tau u^{(\tau)})$ is uniformly bounded in
$L^2(0,T;H^1(\dom)')$. 

{\em Step 3: Limit $\eps\to 0$ and $\tau\to 0$.}
The uniform estimates from Step 2 allow us to apply the Aubin--Lions lemma
in the version of \cite{DrJu12}, which provides the existence of a subsequence
of $(u^{(\tau)})$, which is not relabeled, such that, as $(\eps,\tau)\to 0$,
$$
  u^{(\tau)}\to u \quad\mbox{strongly in }L^1(\dom\times(0,T))\ \Prob\mbox{-a.s.}
$$
In view of the uniform $L^\infty$ bound, this convergence holds in any
$L^p(\dom\times(0,T))$ for $p<\infty$ and a.e.\ in $\dom\times(0,T)$ $\Prob$-a.s.
This allows us to identify the nonlinear weak limits. Moreover, by
weak compactness, $\Prob$-a.s.
\begin{align*}
  \na u^{(\tau)} \rightharpoonup \na u &\quad\mbox{weakly in }L^2(0,T;L^2(\dom)), \\
	\tau^{-1}(u^{(\tau)}-\Gamma_\tau u^{(\tau)})
	\rightharpoonup \pa_t u&\quad\mbox{weakly in }L^2(0,T;H^1(\dom)'), \\
	\eps w^{(\tau)} \to 0 &\quad\mbox{strongly in }L^2(0,T;{L^2(\dom)}).
\end{align*}
Performing the limit $(\eps,\tau)\to 0$ in \eqref{2.weaku} shows that 
$u^{(N,\eta)}:=u$ solves \eqref{2.approx} for all test functions
$\phi\in L^2(0,T;H^1(\dom))$ (by density). We verify as in \cite{Jue15} that $u$ satisfies the initial condition \eqref{2.ic}.
\end{proof}

The proof of \cite[Theorem 2]{Jue15} provides some a priori estimates through the
entropy inequality, but they depend on $\eta$ because of the dependence of
the source term on $\eta$. We derive some uniform bounds in Section \ref{sec.unif}.

Next, we show that the Wong--Zakai approximations converge to the strong solution to
\eqref{2.approx1}--\eqref{2.approx2}. The key consequence is the $L^\infty$ bound for
the solution to \eqref{2.approx1}--\eqref{2.approx2}.

\begin{proposition}\label{prop.uN}
Let $u^{(N,\eta)}$ be the solution  to \eqref{2.approx}--\eqref{2.ic},
constructed in Proposition \ref{prop.Neta},
and let $u^{(N)}$ be the unique strong {(in the probabilistic sense)} solution
to \eqref{2.approx1}--\eqref{2.approx2}, proved in Proposition \ref{prop.ex}.
Then $u^{(N,\eta)}\to u^{(N)}$ in probability up to a stopping time
$\tau_R=\inf\{t\in[0,T]:\|u^{(N)}(t)\|_{H^1(\dom)}>R\}$ as $\eta\to 0$ ($M \to \infty$).
Moreover, it holds that
$u^{(N)}(x,t)\in\overline{\DD}$ for a.e. $(x,t)\in\dom\times(0,T)$ $\Prob$-a.s.
\end{proposition}

\begin{proof}
The result is a consequence of Theorem \ref{thm.wong} in the Appendix. 
We can apply this
theorem since the right-hand side of \eqref{2.approx} is Lipschitz continuous
and has linear growth in $u^{(N,\eta)}$ (see the proof of Proposition \ref{prop.ex}).
\end{proof}


\subsection{Uniform estimates}\label{sec.unif}

We prove some estimates uniform in the approximation parameter $N$. 
The starting point is a stochastic version of the entropy inequality.

\begin{lemma}\label{lem.uNglobal}
The solution $u^{(N)}$ to \eqref{2.approx1}--\eqref{2.approx2} is global-in-time
and satisfies the a priori estimate
$$
  \E\int_\dom h(u^{{(N)}}(t))\,\,\textnormal{d}x + C_1\E\int_0^t\int_\dom\sum_{i=1}^n
	\big|\na (u_i^{(N)})^{1-m}\big|^2 \,\,\textnormal{d}x\,\textnormal{d}s \le C_2,
$$
where $C_1$, $C_2>0$ are independent of $N$ and $R$.
\end{lemma}

\begin{proof}
Let $u^{(N)}$ be the solution to \eqref{2.approx1}--\eqref{2.approx2} up
to the stopping time $\tau_R$. Since the entropy density $h$, defined in Assumption
(A5i), may be not a $C^2$ function on $\overline{\DD}$, we cannot apply the It\^o lemma
to this function. Therefore, we need to regularize $h$. Let us recall the notation
from Assumption (A7): Let $\delta>0$ and 
define $[u]_\delta=([u_1]_\delta,\ldots,[u_n]_\delta)$, where
$$
  [u_i]_\delta := \frac{u_i+\delta/n}{1+\delta}, \ i=1,\ldots,n, \quad
	[u_{n+1}]_\delta := \frac{u_{n+1}}{1+\delta},
$$
and $u_{n+1}=1-\sum_{i=1}^n u_i$. 
Then $[u_{n+1}]_\delta=1-\sum_{i=1}^{n}[u_i]_\delta$ and
$[u]_\delta\in\DD$ for any $u\in\overline{\DD}$. It follows that
$h_\delta(u):=h([u]_\delta)$ satisfies $h_\delta\in C^2(\overline{\DD};[0,\infty))$.

We can now apply the It\^o lemma to $h_\delta$. It holds for $t\in[0,T\wedge\tau_R]$
that 
\begin{align*}
  &\int_\dom h_\delta(u^{(N)}(t\wedge\tau_R))\,\textnormal{d}x
	= \int_\dom h_\delta(u^{(N)}(0))\,\textnormal{d}x \nonumber \\
	&\phantom{xx}{}
	- \int_0^{t\wedge\tau_R}\int_\dom \na u^{(N)}(s):h_\delta''(u^{(N)}(s))A(u^{(N)}(s))
	\na u^{(N)}(s)\,\textnormal{d}x\,\textnormal{d}s \nonumber \\
	&\phantom{xx}{}+ \int_0^{t\wedge\tau_R}\int_\dom(\sigma(u^{(N)}(s))\,\textnormal{d}W(s))\cdot 
	h_\delta'(u^{(N)}(s))\,\textnormal{d}x \\
	&\phantom{xx}{}+ \frac12\int_0^{t\wedge\tau_R}\int_\dom h_\delta'(u^{(N)}(s))\cdot
	\mathcal{T}(u^{(N)}(s))\,\textnormal{d}x\,\textnormal{d}s \nonumber \\ 
	&\phantom{xx}{}+ \frac12\int_0^{t\wedge\tau_R}\int_\dom\operatorname{Tr}
	\Big(\sigma(u^{(N)}(s))h''_\delta(u^{(N)}(s))\sigma(u^{(N)}(s))^\ast\Big)\,\textnormal{d}x\,\textnormal{d}s.
\end{align*}
Taking the expectation on both sides and observing that the expectation of the
It\^o integral vanishes,  we find that
\begin{align}
  & \E\int_\dom h_\delta(u^{(N)}(t\wedge\tau_R))\,\textnormal{d}x
	= \int_\dom h_\delta(u^{(N)}(0))\,\textnormal{d}x \nonumber \\
	&\phantom{xx}{}
	- \E\int_0^{t\wedge\tau_R}\int_\dom \na u^{(N)}(s):h_\delta''(u^{(N)}(s))
	A(u^{(N)}(s))\na u^{(N)}(s)\,\textnormal{d}x\,\textnormal{d}s \nonumber \\
	&\phantom{xx}{}+ \frac12\E\int_0^{t\wedge\tau_R}\int_\dom h_\delta'(u^{(N)}(s))\cdot
	\mathcal{T}(u^{(N)}(s))\,\textnormal{d}x\,\textnormal{d}s \nonumber \\
	&\phantom{xx}{}+ \frac12\E\int_0^{t\wedge\tau_R}\int_\dom\operatorname{Tr}
	\Big(\sigma(u^{(N)}(s))h''_\delta(u^{(N)}(s))\sigma(u^{(N)}(s))^\ast\Big)\,\textnormal{d}x\,\textnormal{d}s 
	\nonumber \\
	&=: J_1^{(\delta)}+\cdots+J_4^{(\delta)}. \label{2.J14}
\end{align}

Our aim is to perform the limit $\delta\to 0$ in \eqref{2.J14}. 
We know that the function $h$ is continuous on $\overline{\DD}$ and
that, as $\delta\to 0$, 
$$
  [u^{(N)}(\omega,x,t\wedge\tau_R)]_\delta\to u^{(N)}(\omega,x,t\wedge\tau_R) 
	\quad\mbox{for a.e. }(\omega,x,t)\in\Omega\times\dom\times(0,T).
$$ 
This implies that  
$$
  h_\delta(u^{(N)}(\omega,x,t\wedge\tau_R))
	= h([u^{(N)}(\omega,x,t\wedge\tau_R)]_\delta) 
	\to h(u^{(N)}(\omega,x,t\wedge\tau_R))
$$
for a.e.\ $(\omega,x,t)\in\Omega\times\dom\times(0,T)$.
Moreover, the integral $\E\int_\dom h_\delta(u^{(N)}(t\wedge\tau_R))\,\textnormal{d}x$ is
uniformly bounded in $\delta$ (since $h$ is bounded on $\overline{\DD}$ by assumption).
We conclude from the dominated convergence theorem that, as $\delta\to 0$,
\begin{align*}
  \E\int_\dom h_\delta(u^{(N)}(t\wedge\tau_R))\,\textnormal{d}x 
	&\to \E\int_\dom h(u^{(N)}(t\wedge\tau_R))\,\textnormal{d}x, \\
	J_1^{(\delta)} = \E\int_\dom h_\delta(u^{(N)}(0))\,\textnormal{d}x &\to \E\int_\dom h(\Pi_N(u^0))\,\textnormal{d}x.
\end{align*}

By Assumption (A7), we have
\begin{align}
  J_2^{(\delta)} &= -\frac{1}{(1+\delta)^2}\E\int_0^{t\wedge\tau_R}\int_\dom
	\na u^{(N)}(s):h''([u^{(N)}(s)]_\delta)A(u^{(N)}(s))\na u^{(N)}(s)\,\textnormal{d}x\,\textnormal{d}s \nonumber \\
	&\le -\frac{c_h}{(1+\delta)^2}\E\int_0^{t\wedge\tau_R}\int_\dom\sum_{i=1}^n 
	\frac{|\na u_i^{(N)}(s)|^2}{[u_i^{(N)}(s)]_\delta^{2m}}\,\textnormal{d}x\,\textnormal{d}s \nonumber \\
	&\phantom{xx}{}- \frac{1}{(1+\delta)^2}\E\int_0^{t\wedge\tau_R}\int_\dom
	\na u^{(N)}(s):R_\delta(u^{(N)}(s))\na u^{(N)}(s) \,\textnormal{d}x\,\textnormal{d}s. \label{2.J2}
\end{align}
Since $R_\delta(u^{(N)}(s))\to 0$ as $\delta\to 0$ uniformly in $u^{(N)}(s)$
and $\na u^{(N)}(s)$ is bounded in $L^2(\dom)$, 
the last integral tends to zero as $\delta\to 0$. Because of
$$
  \frac{1}{(1+\delta)^2}\frac{|\na u_i^{(N)}|^2}{[u_i^{(N)}]_\delta^{2m}}
	\nearrow \frac{|\na u_i^{(N)}|^2}{(u_i^{(N)})^{2m}}
	= \frac{|\na (u_i^{(N)})^{1-m}|^2}{(1-m)^2}
$$
as $\delta\to 0$, the monotone convergence theorem implies that
$$
  \E\int_0^{t\wedge\tau_R}\int_\dom\sum_{i=1}^n
	\frac{|\na u_i^{(N)}(s)|^2}{[u_i^{(N)}]_\delta^{2m}}\,\textnormal{d}x\,\textnormal{d}s
	\to \frac{1}{(1-m)^2}\E\int_0^{t\wedge\tau_R}\int_\dom\sum_{i=1}^n
	|\na(u_i^{(N)})^{1-m}|^2 \,\textnormal{d}x\,\textnormal{d}s
$$
and we infer from \eqref{2.J2} that
$$
  \lim_{\delta\to 0}J_2^{(\delta)} \le -\frac{c_h}{(1-m)^2}\E\int_0^{t\wedge\tau_R}
	\int_\dom\sum_{i=1}^n|\na (u_i^{(N)})^{1-m}|^2 \,\textnormal{d}x\,\textnormal{d}s.
$$

The following a.e.\ pointwise limits hold:
\begin{align*}
  h_\delta'(u^{(N)}(s))\cdot\mathcal{T}(u^{(N)}(s))
	&= \frac{1}{1+\delta}\sum_{i,j,k=1}^n\sigma_{kj}(u^{(N)}(s))
	\frac{\pa\sigma_{ij}}{\pa u_k}(u^{(N)}(s))
	\frac{\pa h}{\pa u_i}([u^{(N)}(s)]_\delta) \\
	&\to h'(u^{(N)})\cdot\mathcal{T}(u^{(N)}(s))
\end{align*}
and
\begin{align*}
	\operatorname{Tr}&\Big(\sigma(u^{(N)}(s))h''_\delta(u^{(N)}(s))
	\sigma(u^{(N)}(s))^\ast\Big) \\
	&= \frac{1}{(1+\delta)^2}\sum_{i,j,k=1}^n\sigma_{ik}(u^{(N)}(s))
	\frac{\pa^2 h}{\pa u_i\pa u_j}([u^{(N)}(s)]_\delta)\sigma_{jk}(u^{(N)}(s)) \\
	&\to \operatorname{Tr}\Big(\sigma(u^{(N)}(s))h''(u^{(N)}(s))
	\sigma(u^{(N)}(s))^\ast\Big)
\end{align*}
for a.e.\ $\Omega\times\dom\times[0,T\wedge\tau_R]$.
Then the bounds imposed in Assumption (A6) imply by dominated convergence 
that these expressions converge in $L^1(\Omega\times\dom\times[0,T\wedge\tau_R])$, 
which means that
\begin{align*}
  J_3^{(\delta)} &\to \frac12\E\int_0^{t\wedge\tau_R}\int_\dom h'(u^{(N)}(s))
	\cdot\mathcal{T}(u^{(N)}(s))\,\textnormal{d}x\,\textnormal{d}s, \\
  J_4^{(\delta)} &\to \frac12\E\int_0^{t\wedge\tau_R}\int_\dom\operatorname{Tr}
	\Big(\sigma(u^{(N)}(s))h''(u^{(N)}(s))\sigma(u^{(N)}(s))^\ast\Big)\,\textnormal{d}x\,\textnormal{d}s.
\end{align*}
Using Assumption (A6) again, we see that the limits of $J_3^{(\delta)}$ and
$J_4^{(\delta)}$ are bounded with respect to $N$ and $R$, 
and Assumption (A2) implies that
the limit of $J_1^{(\delta)}$ is uniformly bounded in $N$. Then the limit
$\delta\to 0$ in \eqref{2.J14} yields the entropy inequality
\begin{align*}
  \E\int_\dom & h(u^{{(N)}}(t\wedge\tau_R))\,\textnormal{d}x
	+ C_1\E\int_0^{t\wedge\tau_R}\int_\dom\sum_{i=1}^n
	|\na (u_i^{(N)}(s))^{1-m}|^2 \,\textnormal{d}x\,\textnormal{d}s \\
	&\le \E\int_\dom h(\Pi_N(u^0))\,\textnormal{d}x + C_2,
\end{align*}
where the constants $C_1>0$ and $C_2>0$ are independent of $N$ and $R$. Consequently,
the right-hand side of this inequality does not depend on the chosen sequence
of stopping times $\tau_R$, and we can pass to the limit $R\to\infty$. Hence, the previous inequality
holds for any $t\in[0,T]$.
The uniform $L^\infty$ estimate implies that
\begin{equation}\label{2.estL2}
  \sup_{N\in\N}\E\bigg(\sup_{0<t<T}\|u^{(N)}(t)\|_{L^2(\dom)}^p\bigg)
	\le C(T,u^0)
\end{equation}
for all $1\le p<\infty$, and the entropy inequality shows that
\begin{equation}\label{2.estH1}
  \sup_{N\in\N}\E\|u^{(N)}\|_{L^2(0,T;H^1(\dom))}^2 \le C(T,u^0),
\end{equation}
where $C(T,u^0)>0$ is independent of $N$, since 
\begin{align*}
  \E\|\na u_i^{(N)}\|^2_{L^2(0,T;L^2(\dom))}
	&= \int_\Omega \int_0^T \int_\dom \left|\frac{1}{1-m}(u_i^{(N)})^{m}\na (u_i^{(N)})^{1-m}\right|^2 \rm{d}x\,\rm{d}t\,\rm{d}\Prob(\omega) \\
	& =  \frac{1}{(1-m)^2} \int_\Omega \int_0^T \int_\dom |u_i^{(N)}|^{2m}
	|\na (u_i^{(N)})^{1-m}|^2\rm{d}x\,\rm{d}t\rm{d}\Prob(\omega)\\
	&\le \frac{1}{(1-m)^2} \int_\Omega \int_0^T \int_\dom 	|\na (u_i^{(N)})^{1-m}|^2\rm{d}x\,\rm{d}t\rm{d}\Prob(\omega) \\
	& = \frac{1}{(1-m)^2}\E\|\na(u_i^{(N)})^{1-m}\|^2_{L^2(0,T; L^2(\dom))}\le C.
\end{align*} 
Here, we used $|u^{(N)}(\omega,x,t)| \le 1$ for almost all $(\omega,x,t) \in \Omega\times \dom \times[0,T]$ and that $m<1$. Since $T>0$ was arbitrary, 
the solution $u^{(N)}$ to \eqref{2.approx1}--\eqref{2.approx2} is global-in-time. 
\end{proof}


\subsection{Tightness of the laws of $(u^{(N)})$}\label{sec.tight}

Let $u^{(N)}$ be a solution to \eqref{2.approx1}--\eqref{2.approx2}, constructed
in Lemma \ref{lem.uNglobal}. 
We show that the laws of $u^{(N)}$ are tight in a certain sub-Polish space.
(This is a topological space in which there exists a countable family of 
continuous functions that separate points \cite[Definition 2.1.3]{BFH18}.)
For this, we proceed similarly as in \cite{DJZ19} and introduce the following
spaces:
\begin{itemize}
\item $C^0([0,T];H^3(\dom)')$ is the space of continuous functions
$u:[0,T]\to H^3(\dom)'$ with the topology $\mathbb{T}_1$ induced by the norm
$\|u\|_{C^0([0,T];H^3(\dom)')}=\sup_{t\in(0,T)}\|u(t)\|_{H^3(\dom)'}$;
\item $L^2_w(0,T;H^1(\dom))$ is the space $L^2(0,T;H^1(\dom))$ with the weak
topology $\mathbb{T}_2$;
\item $L^2(0,T;L^2(\dom))$ is the space of square integrable functions 
$u:(0,T)\to L^2(\dom)$ with the topology $\mathbb{T}_3$ induced by the norm
$\|\cdot\|_{L^2(0,T;L^2(\dom))}$;
\item $C^0([0,T];L^2_w(\dom))$ is the space of weakly continuous functions
$u:[0,T]\to L^2(\dom)$ endowed with the weakest topology $\mathbb{T}_4$ such that
for all $ \psi \in L^2(\dom)$, the mappings
$$
  C^0([0,T];L_w^2(\dom))\to C^0([0,T];\R), \quad u\mapsto {\langle u(\cdot),\psi\rangle},
$$
are continuous.
\end{itemize}
We define the space
\begin{equation*}
  Z_T:=C^0([0,T];H^3(\dom)')\cap L_w^2(0,T;H^1(\dom))\cap L^2(0,T;L^2(\dom))
	\cap C^0([0,T];L_w^2(\dom)),
\end{equation*}
endowed with the topology $\mathbb{T}$ that is the maximum of the topologies
$\mathbb{T}_i$, $i=1,2,3,4$, of the corresponding spaces. It is shown in
\cite[Lemma 12]{DJZ19} that $Z_T$ is a sub-Polish space.

\begin{lemma}\label{lem.tight}
The set of laws $(\mathcal{L}(u^{(N)}))_{N\in\N}$ is tight in $Z_T$.
\end{lemma}

\begin{proof}
The idea is to apply the tightness criterion of Brze\'{z}niak and Motyl 
\cite[Corollary  2.6]{BrMo14} with the spaces $U=H^3(\dom)$, $V=H^1(\dom)$, and
$H=L^2(\dom)$ (also see the proof of Lemma 11 in \cite{DJZ19}). Estimates
\eqref{2.estL2} and \eqref{2.estH1} are exactly conditions (a) and (b) in
\cite{BrMo14}. It remains to show that $(u^{(N)})_{N\in\N}$ satisfies the Aldous
condition in $H^3(\dom)'$. We need to show that for any $\eps>0$ and $\kappa>0$, 
there exists $\theta_0>0$ such that for any sequence $(\tau_N)_{N\in\N}$
of $\mathbb{F}$-stopping times, it holds that
$$
  \sup_{N\in\N}\sup_{0<\theta<\theta_0}\Prob\big\{\|u^{(N)}(\tau_N+\theta)
	-u^{(N)}(\tau_N)\|_{H^3(\dom)'}\ge\kappa\big\} \le\eps.
$$
We proceed similarly as in \cite[Lemma 11]{DJZ19}. Let $(\tau_N)_{N\in\N}$
be a sequence of $\mathbb{F}$-stopping times such that $0\le\tau_N\le T$
and let $t\in[0,T]$ and $\phi\in H^3(\dom)$. The solution $u^{(N)}$ to
\eqref{2.approx1}--\eqref{2.approx2} solves
\begin{align}
  \langle u_i^{(N)}(t),\phi\rangle
	&= \langle\Pi_N(u_i^0),\phi\rangle
	- \int_0^t\sum_{j=1}^n\big\langle A_{ij}(u^{(N)})\na u_j^{(N)},\na\Pi_N\phi\rangle \,\textnormal{d}s 
	\nonumber \\
	&\phantom{xx}{}+ \frac12\int_0^t\langle \Pi_N(\mathcal{T}_i(u^{(N)}),\phi\rangle \,\textnormal{d}s
	+ \bigg\langle\int_0^t\sum_{j=1}^n\Pi_N(\sigma_{ij}(u^{(N)}))\,\textnormal{d}W_j,\phi\bigg\rangle 
	\nonumber \\
	&=: J_1^{(N)} + J_2^{(N)}(t) + J_3^{(N)}(t) + J_4^{(N)}(t). \label{2.J1234}
\end{align}

Consider first the term involving the diffusion coefficients. Let $\theta>0$.
We use assumption (A3), the continuous embedding $H^3(\dom)\hookrightarrow 
W^{1,\infty}(\dom)$ (for $d\le 3$), and estimates \eqref{2.estL2}--\eqref{2.estH1} 
to find that 
\begin{align*}
  \E\bigg|&\int_{\tau_N}^{\tau_N+\theta}\big\langle A_{ij}(u^{(N)})\na u_j^{(N)},
	\na\Pi_N\phi\rangle \,\textnormal{d}s\bigg| \\
	&\le C\E\int_{\tau_N}^{\tau_N+\theta}\big(1+\|u^{(N)}\|_{L^2(\dom)}\big)
	\|\na u^{(N)}\|_{L^2(\dom)}\|\na\phi\|_{L^\infty(\dom)}\,\textnormal{d}s \\
	&\le C\theta^{1/2}\E\big((1+\|u^{(N)}\|_{L^\infty(0,T;L^2(\dom))})
	\|\na u^{(N)}\|_{L^2(0,T;L^2(\dom))}\big)\|\phi\|_{H^3(\dom)} \\
	&\le C\theta^{1/2}\bigg\{1 + \E\bigg(\sup_{0<t<T}\|u^{(N)}(t)\|_{L^2(\dom)}^2\bigg)
	\bigg\}^{1/2}\bigg\{\E\int_0^T\|\na u^{(N)}\|_{L^2(\dom)}^2 \,\textnormal{d}s\bigg\}^{1/2}
	\|\phi\|_{H^3(\dom)} \\
	&\le C\theta^{1/2}\|\phi\|_{H^3(\dom)},
\end{align*}
where we applied first the Cauchy--Schwarz inequality with respect to time and
then with respect to the random variable. For the It\^o correction term, 
we use the boundedness of $u^{(N)}$ and the Cauchy--Schwarz inequality:
\begin{align*}
  \E\bigg|\int_{\tau_N}^{\tau_N+\theta}\langle\Pi_N(\mathcal{T}_i(u^{(N)}),\phi
	\rangle \,\textnormal{d}s\bigg|
	&\le \E\int_{\tau_N}^{\tau_N+\theta}\|\mathcal{T}_i(u^{(N)})\|_{L^2(\dom)}
	\|\phi\|_{L^2(\dom)} \\
	&\le C\theta^{1/2}\|\phi\|_{H^3(\dom)}.
\end{align*}
For the stochastic term, we take into account Assumption (A4), the It\^o isometry, and
again the Cauchy--Schwarz inequality:
\begin{align*}
  \E\bigg|&\bigg\langle\int_{\tau_N}^{\tau_N+\theta}\Pi_N(\sigma_{ij}(u^{(N)}))
	\,\textnormal{d}W_j,\phi\bigg\rangle \,\textnormal{d}s\bigg|^2 \\
	&\le \E\int_{\tau_N}^{\tau_N+\theta}
	\|\sigma(u^{(N)}(s))\|_{\mathcal{L}_2(\R^n;L^2(\dom))}^2 \,\textnormal{d}s\|\phi\|_{L^2(\dom)}^2 \\
	&\le \E\int_{\tau_N}^{\tau_N+\theta}\big(1+\|u^{(N)}(s)\|_{L^2(\dom)}^2\big)\,\textnormal{d}s
	\|\phi\|_{L^2(\dom)}^2 \\
	&\le C\bigg\{\theta + \theta^{1/3}\bigg(E\int_0^T\|u^{(N)}(s)\|_{L^2(\dom)}^3 \,\textnormal{d}s
	\bigg)^{2/3}\bigg\}\|\phi\|_{L^2(\dom)}^2 \le C\theta^{1/3}\|\phi\|_{H^3(\dom)}^2.
\end{align*}

Note that the previous estimates could be simplified since $u^{(N)}$ is 
uniformly bounded. Our estimates hold under minimal requirements and may be
used for generalizations.

Let $\kappa>0$ and $\eps>0$. In view of the previous estimates and using the
Chebyshev inequality, it follows for $i=2,3$ that
\begin{align*}
  \Prob\big\{&\|J_i^{(N)}(\tau_N+\theta)-J_i^{(N)}(\tau_N)\|_{H^3(\dom)'}
	\ge\kappa\big\}
	\le \frac{1}{\kappa}\E\big\|J_i^{(N)}(\tau_N+\theta)-J_i^{(N)}(\tau_N)
	\big\|_{H^3(\dom)'} \\
	&= \frac{1}{\kappa}\sup_{\|\phi\|_{H^3(\dom)}=1}\E\big|\big\langle
	J_i^{(N)}(\tau_N+\theta)-J_i^{(N)}(\tau_N),\phi\big\rangle\big|
	\le \frac{C\theta^{1/2}}{\kappa},
\end{align*}
while for $i=4$, we have
\begin{align*}
  \Prob\big\{&\|J_4^{(N)}(\tau_N+\theta)-J_4^{(N)}(\tau_N)\|_{H^3(\dom)'}
	\ge\kappa\big\} \\
	&\le \frac{1}{\kappa}\sup_{\|\phi\|_{H^3(\dom)}=1}\E\big|\big\langle
	J_i^{(N)}(\tau_N+\theta)-J_i^{(N)}(\tau_N),\phi\big\rangle\big| \\
  &\le \frac{C}{\kappa}\sup_{\|\phi\|_{H^3(\dom)}=1}\Big(\E\big|\big\langle
	J_i^{(N)}(\tau_N+\theta)-J_i^{(N)}(\tau_N),\phi\big\rangle\big|^2\Big)^{1/2}
	\le \frac{C\theta^{1/6}}{\kappa}.
\end{align*}
Thus, choosing $\theta_0=\min\{1,(\kappa\eps/C)^6\}$, we infer that for $i=2,3,4$,
$$
  \sup_{N\in\N}\sup_{0<\theta<\theta_0}
	\Prob\big\{\|J_i^{(N)}(\tau_N+\theta)-J_i^{(N)}(\tau_N)\|_{H^3(\dom)'}\ge\kappa\big\}
	\le \eps.
$$
This shows that the Aldous condition holds for all three terms $J_i^{(N)}$ $(i=2,3,4$)
and consequently, in view of \eqref{2.J1234}, also for $(u_i^{(N)})_{N\in\N}$.
Thus, by \cite[Corollary 2.6]{BrMo14}, the set of laws of $(u^{(N)})_{N\in\N}$
is tight in $Z_T$.
\end{proof}


\subsection{Convergence of $(u^{(N)})_{N\in\N}$}\label{sec.conv}

Since $Z_T\times C^0([0,T];\R^n)$ satisfies the assumptions of the 
Skorokhod--Jakubowski theorem \cite[Theorem C1]{BrOn10} and the sequence of laws of 
$(u^{(N)})_{N\in\N}$ is tight on $(Z_T,\mathbb{T})$ by Lemma \ref{lem.tight}, 
this theorem implies the existence of a subsequence of $(u^{(N)})_{N\in\N}$,
which is not relabeled, a probability space $(\widetilde\Omega,
\widetilde{\mathcal{F}},\widetilde\Prob)$ and, on this space,
$(Z_T\times C^0([0,T];\R^n))$-valued random variables $(\widetilde u,\widetilde W)$
and $(\widetilde u^{(N)},\widetilde W^{(N)})$ for $N\in\N$ such that
$(\widetilde u^{(N)},\widetilde W^{(N)})$ has the same law as $(u^{(N)},W)$
on $\mathcal{B}(Z_T\times C^0([0,T];\R^n))$ and, as $N\to\infty$,
$$
  (\widetilde u^{(N)},\widetilde W^{(N)})\to (\widetilde u,\widetilde W)
	\quad\mbox{in }Z_T\times C^0([0,T];\R^n)\ \widetilde\Prob\mbox{-a.s.}
$$
Because of the definition of the space $Z_T$, this convergence means 
$\widetilde\Prob$-a.s.,
\begin{align}
  \widetilde u^{(N)}\to \widetilde u &\quad\mbox{in }C^0([0,T];H^3(\dom)'), \nonumber \\
	\widetilde u^{(N)}\rightharpoonup \widetilde u &\quad\mbox{weakly in }
	L^2(0,T;H^1(\dom)), \nonumber \\
	\widetilde u^{(N)}\to \widetilde u &\quad\mbox{in }L^2(0,T;L^2(\dom)), 
	\label{2.conv} \\
	\widetilde u^{(N)}\to \widetilde u &\quad\mbox{in }C^0([0,T];L_w^2(\dom)), \nonumber \\
  \widetilde W^{(N)}\to \widetilde W &\quad\mbox{in }C^0([0,T];\R^n). \nonumber
\end{align}

We wish to derive some regularity properties for the limit $\widetilde u$.
To this end, we proceed as in \cite[Section 2.5]{DJZ19}.
Since $u^{(N)}$ is an element of $C^0([0,T];H_N)$ $\Prob$-a.s., $C^0([0,T];H_N)$
is a Borel set of $C^0([0,T];H^3(\dom)')\cap L^2(0,T;L^2(\dom))$, and
$u^{(N)}$ and $\widetilde u^{(N)}$ have the same law {on $\mathcal{B}(Z_T)$}, we infer that
$$
  \mathcal{L}(\widetilde u^{(N)})\big(C^0([0,T];H_N)\big) = 1
	\quad\mbox{for all }N\in\N.
$$
Observe that $\widetilde u$ is a $Z_T$-Borel random variable since
$\mathcal{B}(Z_T\times C^0([0,T];\R^n))$ is a subset of $\mathcal{B}(Z_T)\times
\mathcal{B}(C^0([0,T];\R^n))$. Furthermore, estimates \eqref{2.estL2}--\eqref{2.estH1}
and the equivalence of the laws of $\widetilde u{(N)}$ and $u{(N)}$ on
$\mathcal{B}(Z_T)$ yield for any $p\ge 1$ the following uniform estimates:
\begin{align*}
  \sup_{N\in\N}\widetilde\E\bigg(\sup_{0<t<T}\|\widetilde u^{(N)}(t)\|_{L^\infty(\dom)}^p
	\bigg) &\le C, \\
	\sup_{N\in\N}\widetilde\E\bigg(\int_0^T\|\widetilde u^{(N)}(t)\|_{H^1(\dom)}^2 \,\textnormal{d}t
	\bigg) &\le C.
\end{align*}
We deduce the existence of a subsequence of $(\widetilde u^{(N)})_{N\in\N}$
(not relabeled) which is weakly* converging in $L^p(\widetilde\Omega;L^\infty(0,T;
L^\infty(\dom)))$ and weakly converging in {$L^2(\widetilde\Omega;L^2(0,T;
H^1(\dom)))$} as $N\to\infty$. Since $\widetilde u^{(N)}\to\widetilde u$
in $Z_T$ $\widetilde\Prob$-a.s., we conclude that 
$\widetilde u\in {L^p(\widetilde\Omega;L^\infty(0,T;L^\infty(\dom)))}$ for any $p\ge 1$
and $\widetilde u\in L^2(\widetilde\Omega;L^2(0,T;H^1(\dom)))$, i.e.
$$
  \widetilde\E\bigg(\sup_{0<t<T}\|\widetilde u(t)\|_{L^\infty(\dom)}^p\bigg) < \infty, 
	\quad \widetilde\E\int_0^T\|\widetilde u(t)\|_{H^1(\dom)}^2 \,\textnormal{d}t < \infty.
$$

We claim that $\widetilde u$ is even bounded in $\overline{\DD}$ $\widetilde\Prob$-a.s.

\begin{lemma}\label{lem.bound}
The limit $\widetilde u$ satisfies $\widetilde u(x,t)\in\overline\DD$ for a.e.\
$(x,t)\in\dom\times(0,T)$ $\widetilde\Prob$-a.s.
\end{lemma}

\begin{proof}
By Proposition \ref{prop.uN}, $u^{(N)}(x,t)\in\overline{\DD}$ 
for a.e.\ $(x,t)\in\dom\times(0,T)$ $\Prob$-a.s.
In particular,
\begin{equation}\label{2.le1}
  \|u^{(N)}\|_{L^\infty(0,T;L^\infty(\dom))} 
	:= \sum_{i=1}^n\|u_i^{(N)}\|_{L^\infty(0,T;L^\infty(\dom))} \le 1.
\end{equation}
The set $L^\infty(0,T;L^\infty(\dom))$ is continuously embedded in
$L^\infty(0,T;H^3(\dom)')\cap L^2(0,T;L^2(\dom))$. {Thus, by the
Kuratowski theorem (see Theorem \ref{thm.kura} in the Appendix),
$L^\infty(0,T;L^\infty(\dom))$ is a Borel set of 
$L^\infty(0,T;H^3(\dom)')\cap L^2(0,T;L^2(\dom))$. Therefore, by \cite[Lemma B.1]{BrDh19}, the set $L^\infty(0,T;$ $L^\infty(\dom))\cap Z_T$
is a Borel subset of $L^\infty(0,T;H^3(\dom)')\cap L^2(0,T;L^2(\dom)) \cap Z_T$ which in turn is $Z_T$.} \dela{Thus, by the
Kuratowski theorem (see Theorem \ref{thm.kura} in the Appendix),
$L^\infty(0,T;L^\infty(\dom))$ is a Borel set of 
$L^\infty(0,T;H^3(\dom)')\cap L^2(0,T;L^2(\dom))$ and, in fact, also of
$C^0([0,T];H^3(\dom)')\cap L^2(0,T;L^2(\dom))$ (since the norms are the same). 
By \cite[Lemma B.1]{BrDh19}, the set $L^\infty(0,T;$ $L^\infty(\dom))\cap Z_T$
is a Borel subset of $Z_T$. }
The equivalence of the laws of $\widetilde u^{(N)}$ and $u^{(N)}$ on $\mathcal{B}(Z_T)$ 
as well as \eqref{2.le1} then show that
$$
  \widetilde\Prob\big\{\|\widetilde u^{(N)}\|_{L^\infty(0,T;L^\infty(\dom))}\le 1
	\big\} = \Prob\big\{\|u^{(N)}\|_{L^\infty(0,T;L^\infty(\dom))}\le 1\big\} = 1.
$$
By the definition of the norm in \eqref{2.le1}, this means that
\begin{equation}\label{2.B1}
  \sum_{i=1}^n|\widetilde u_i^{(N)}(x,t)| \le 1 \quad\mbox{for a.e. }
	(x,t)\in\dom\times(0,T)\ \widetilde\Prob\mbox{-a.s.}
\end{equation}

Next, we show that $\widetilde u_i^{(N)}(x,t)\ge 0$ for a.e.\ $(x,t)\in\dom\times(0,T)$
$\widetilde\Prob$-a.s. Let $v\in L^\infty(0,T;$ $L^\infty(\dom))$ 
and define the closed unit ball
$$
  B(v) = \big\{u\in L^\infty(0,T;L^\infty(\dom)):
	\|u-v\|_{L^\infty(0,T;L^\infty(\dom))}\le 1\big\}.
$$
We deduce from \eqref{2.B1} that
$$
  \widetilde\Prob\big(\widetilde u_i^{(N)}\in B(0)\big) = 1, \quad i=1,\ldots,n.
$$
Since $0\le u_i^{(N)}(x,t)\le 1$ a.e.\ in $\dom\times(0,T)$ $\Prob$-a.s., we have 
$\|u_i^{(N)}-1\|_{L^\infty(0,T;L^\infty(\dom))}\le 1$ for all $i=1,\ldots,n$
and consequently, by the equivalence of the laws,
$$
  \widetilde\Prob\big(\widetilde u_i^{(N)}\in B(1)\big)
	= \Prob\big(u_i^{(N)}\in B(1)\big) = 1.
$$
We infer that
$$
  \widetilde\Prob\big(\widetilde u_i^{(N)}\in B(1)\cap B(0)\big) = 1,
$$
and this implies that $0\le \widetilde u_i^{(N)}(x,t)\le 1$ $\widetilde\Prob$-a.s.
and, taking into account \eqref{2.B1}, $\sum_{i=1}^n\widetilde u_i^{(N)}(x,t)$ $\le 1$,
i.e.\ $\widetilde u^{(N)}(x,t)\in\overline\DD$ $\widetilde\Prob$-a.s. Moreover, from \eqref{2.conv} we know that $\widetilde u^{(N)}$ converges to $\widetilde u$ strongly in $L^2(0,T; L^2(\dom))$ $\widetilde \Prob$-a.s. and thus we conclude that $\widetilde u (x,t) \in \overline{\DD}$ for a.e. $(x,t) \in \dom \times(0,T)$ $\widetilde\Prob$-a.s.
\end{proof}

We denote by $\widetilde{\mathbb F}$ and $\widetilde{\mathbb F}^{(N)}$ the filtrations 
generated by $(\widetilde u,\widetilde W)$ and $(\widetilde u^{(N)},\widetilde W^{(N)})$,
respectively. Lemmas 14--15 in \cite{DJZ19} imply that $\widetilde u$ is
progressively measurable with respect to $\widetilde{\mathbb F}$ and that
$\widetilde u^{(N)}$ is progressively measurable with respect to 
$\widetilde{\mathbb F}^{(N)}$. 

The following lemma is needed to prove that $(\widetilde u,\widetilde W)$ is a
martingale solution to \eqref{1.eq}--\eqref{1.bic}.

\begin{lemma}\label{lem.E} 
It holds for all $s$, $t\in[0,T]$ with $s\le t$ and all $\phi_1\in {L^{2}(\dom)}$ and
$\phi_2\in H^3(\dom)$ satisfying $\na\phi_2\cdot\nu=0$ on $\pa\dom$ that
\begin{align}
  \lim_{N\to\infty}\widetilde\E\int_0^T\big\langle\widetilde u_{i}^{(N)}(t)-\widetilde u_{i}(t),
	\phi_1\big\rangle^2\,\textnormal{d}t &= 0, \label{3.E1} \\
	\lim_{N\to\infty}\widetilde\E\big\langle\widetilde u_{i}^{(N)}(0)-\widetilde u_{i}(0),
	\phi_1\big\rangle^2 &= 0, \label{3.E2} \\
	\lim_{N\to\infty}\widetilde\E\int_0^T\bigg|\sum_{j=1}^n\int_0^t\Big\langle
	A_{ij}(\widetilde u^{(N)}(s))\na\widetilde u_j^{(N)}(s)
	- A_{ij}(\widetilde u(s))\na\widetilde u_j(s),\na\phi_2\Big\rangle
	\,\textnormal{d}s\bigg|\,\textnormal{d}t &= 0, \label{3.E3} \\
	\lim_{N\to\infty}\widetilde\E\int_0^T\bigg|\int_0^t\Big\langle
	\mathcal{T}_{i}(\widetilde u^{(N)}(s))- \mathcal{T}_{i}(\widetilde u(s)),
	\phi_1\Big\rangle \,\textnormal{d}s\bigg|\,\textnormal{d}t &= 0, \label{3.E4} \\
	\lim_{N\to\infty}\widetilde\E\int_0^T\bigg|\sum_{j=1}^n\int_0^t\Big\langle
	\sigma_{ij}(\widetilde u^{(N)}(s))\textnormal{d}\widetilde W_j^{(N)}(s)-\sigma_{ij}(\widetilde u(s))
	\textnormal{d}\widetilde W_j(s),\phi_1\Big\rangle\bigg|^2 \,\textnormal{d}t &= 0. \label{3.E5}
\end{align}
\end{lemma} 

\begin{proof}
The convergences \eqref{3.E1} and \eqref{3.E2} can be shown as in the proof
of \cite[Lemma 16]{DJZ19}. The convergence \eqref{3.E3} follows from the
Lipschitz continuity of $A_{ij}$ in the bounded domain $\overline\DD$:
\begin{align*}
  \bigg|\int_0^t&\Big\langle
	A_{ij}(\widetilde u^{(N)}(s))\na\widetilde u_j^{(N)}(s)
	- A_{ij}(\widetilde u(s))\na\widetilde u_j(s),\na\phi_2\Big\rangle
	\,\textnormal{d}s\bigg| \\
	&\le \int_0^t\|A_{ij}(\widetilde u^{(N)}(s))-A_{ij}(\widetilde u(s))\|_{L^2(\dom)}
	\|\na \widetilde u_j^{(N)}(s)\|_{L^2(\dom)}\|\na\phi_2\|_{L^\infty(\dom)}\,\textnormal{d}s \\
	&\phantom{xx}{}
	+ \bigg|\int_0^t A_{ij}(\widetilde u(s))\na(\widetilde u^{(N)}(s)-\widetilde u(s))
	\cdot\na\phi_2 \,\textnormal{d}s\bigg|.
\end{align*} 
Since $(\widetilde u^{(N)})$ is bounded in $L^\infty(0,T;L^\infty(\dom))$
$\widetilde\Prob$-a.s.\ and the function $u\mapsto A_{ij}(u)$ is Lipschitz continuous on
bounded sets, the strong $L^2$ convergence of $(\widetilde u^{(N)})$ implies
that $A_{ij}(\widetilde u^{(N)})\to A_{ij}(\widetilde u)$ strongly in
$L^2(0,T;L^2(\dom))$ $\widetilde\Prob$-a.s. Therefore, the first term on the 
right-hand side converges to zero. We deduce from the weak convergence
$\na\widetilde u^{(N)}\to\na\widetilde u$ weakly in $L^2(0,T; L^2(\dom))$ 
$\widetilde\Prob$-a.s.\
that also the second term on the right-hand side converges to zero. This shows that
\begin{equation}\label{2.convA}
  \lim_{N\to\infty}\int_0^t\big\langle A_{ij}(\widetilde u^{(N)}(s))
	\na\widetilde u_j^{(N)}(s),\na\phi_2\big\rangle \,\textnormal{d}s
	= \int_0^t\big\langle A_{ij}(\widetilde u(s))\na\widetilde u_j(s),
	\na\phi_2\big\rangle \,\textnormal{d}s\quad\widetilde\Prob\mbox{-a.s.}
\end{equation}
for all $\phi_2\in H^3(\dom)$ satisfying $\na\phi_2\cdot\nu=0$ on $\pa\dom$. 
We compute 
\begin{align*}
  \widetilde\E\bigg|&\int_0^t\big\langle A_{ij}(\widetilde u^{(N)}(s))
	\na\widetilde u_j^{(N)}(s),\na\phi_2\big\rangle \,\textnormal{d}s\bigg|^{3/2} \\
	&\le \|\na\phi_2\|_{L^\infty(\dom)}^{3/2}\widetilde\E\bigg|\int_0^t
	\big(1+\|\widetilde u^{(N)}(s)\|_{L^2(\dom)}\big)\|\na\widetilde u^{(N)}(s)
	\|_{L^2(\dom)}\,\textnormal{d}s\bigg|^{3/2} \\
	&\le C\|\phi_2\|_{H^3(\dom)}^{3/2}T^{3/4}\widetilde\E\bigg\{
	\big(1+\|\widetilde u^{(N)}\|_{L^\infty(0,T;L^2(\dom))}\big)^{3/2}
	\bigg(\int_0^T\|\na\widetilde u^{(N)}(s)\|_{L^2(\dom)}^2 \,\textnormal{d}s\bigg)^{3/4}\bigg\} \\
	&\le C\|\phi_2\|_{H^3(\dom)}^{3/2}T^{3/4}
	\Big(\widetilde\E\big(1+\|\widetilde u^{(N)}\|_{L^\infty(0,T;L^2(\dom))}^6\big)
	\Big)^{1/4}\Big(\widetilde\E
	\|\widetilde u^{(N)}\|_{L^2(0,T;H^1(\dom))}^2\Big)^{3/4} \le C.
\end{align*}
This bound and the $\widetilde\Prob$-a.s.\ convergence \eqref{2.convA} allow us to
apply the Vitali convergence theorem to infer that \eqref{3.E3} holds.

Analogous arguments lead to the convergence $\mathcal{T}_i(\widetilde u^{(N)})
\to\mathcal{T}_i(\widetilde u)$ strongly in $L^2(0,T;$ $L^2(\dom))$ 
$\widetilde\Prob$-a.s.
(since $\pa\sigma/\pa u_k$ is bounded). Moreover, for $\phi_1\in L^2(\dom)$,
\begin{align*}
  \widetilde\E&\bigg|\int_0^t\big\langle\mathcal{T}_i(\widetilde u^{(N)}(s)),
	\phi_1\big\rangle
	\,\textnormal{d}s\bigg|^2 \le \|\phi_1\|_{L^2(\dom)}^2\widetilde\E\bigg|\int_0^t
	\|\mathcal{T}_i(\widetilde u^{(N)}(s))\|_{L^2(\dom)}\,\textnormal{d}s\bigg|^2 \\
	&\le C\|\phi_1\|_{L^2(\dom)}^2 T\widetilde\E\big(
	1+\|\widetilde u^{(N)}\|_{L^2(0,T;L^2(\dom))}^2\big) \le C,
\end{align*}
and Vitali's convergence theorem implies that \eqref{3.E4} holds.

It remains to prove convergence \eqref{3.E5}. Since $\widetilde W^{(N)}\to
\widetilde W$ in $C^0([0,T];\R^n)$, it is sufficient to show that
$\sigma_{ij}(\widetilde u^{(N)})\to\sigma_{ij}(\widetilde u)$ in 
$L^2(0,T;L^2(\dom))$  $\widetilde\Prob$-a.s. We estimate for
$\phi_1\in L^2(\dom)$, 
\begin{align*}
  \int_0^t&\big|\big\langle\sigma_{ij}(\widetilde u^{(N)}(s))
	-\sigma_{ij}(\widetilde u(s)),\phi_1\big\rangle\big|^2 \,\textnormal{d}s \\
	&\le \int_0^t\big\|\sigma_{ij}(\widetilde u^{(N)}(s))-\sigma_{ij}(\widetilde u(s))
	\big\|_{L^2(\dom)}^2\|\phi_1\|_{L^2(\dom)}^2 \,\textnormal{d}s \\
  &\le C\|\widetilde u^{(N)}(s)-\widetilde u(s)\|_{L^2(0,T;L^2(\dom))}^2
	\|\phi_1\|_{L^2(\dom)}^2.
\end{align*}
Then, by the strong $L^2$ convergence $\widetilde\Prob$-a.s.\ of $(\widetilde u^{(N)})$,
$$
  \lim_{N\to\infty}\int_0^t\big|\big\langle\sigma_{ij}(\widetilde u^{(N)}(s))
	-\sigma_{ij}(\widetilde u(s)),\phi_1\big\rangle\big|^2 \,\textnormal{d}s = 0.
$$
Furthermore,
\begin{align*}
  \widetilde\E&\bigg|\int_0^t\big|\big\langle\sigma_{ij}(\widetilde u^{(N)}(s))
	-\sigma_{ij}(\widetilde u(s)),\phi_1\big\rangle\big|^2 
	\,\textnormal{d}s\bigg|^2 \\
	&\le C\|\phi_1\|_{L^2(\dom)}^4\widetilde\E\int_0^t\big(
	\|\sigma_{ij}(\widetilde u^{(N)}(s))\|_{L^2(\dom)}^4
	+ \|\sigma_{ij}(\widetilde u(s)\|_{L^2(\dom)}^4\big)\,\textnormal{d}s \\
	&\le CT\|\phi_1\|_{L^2(\dom)}^4
	\widetilde E\bigg(\sup_{0<s<T}\|\widetilde u^{(N)}(s)\|_{L^2(\dom)}^4
	+ \sup_{0<s<T}\|\widetilde u(s)\|_{L^2(\dom)}^4\bigg) \le C.
\end{align*}
In view of Vitali's convergence theorem, we deduce from this bound and the 
previous convergence that
$$
  \lim_{N\to\infty}\widetilde\E\int_0^t
	\big|\big\langle\sigma_{ij}(\widetilde u^{(N)}(s))
	- \sigma_{ij}(\widetilde u(s)),\phi_1\big\rangle
	\big|^2 \,\textnormal{d}s = 0.
$$
We deduce from the It\^o isometry that
\begin{equation}\label{2.convsig}
  \lim_{N\to\infty}\widetilde\E\bigg|\bigg\langle\int_0^t\big(
	\sigma_{ij}(\widetilde u^{(N)}(s)-\sigma_{ij}(\widetilde u(s))\big)\textnormal{d}\widetilde W_j(s),
	\phi_1\bigg\rangle\bigg|^2 = 0,
\end{equation}
and we can estimate as
\begin{align*}
  \widetilde\E&\bigg|\bigg\langle\int_0^t\big(\sigma_{ij}(\widetilde u^{(N)}(s))
	- \sigma_{ij}(\widetilde u(s))\big)\textnormal{d}\widetilde W_j(s),\phi_1\bigg\rangle
	\bigg|^2 \\
	&= \widetilde\E\int_0^t\big|\langle\sigma_{ij}(\widetilde u^{(N)}(s))
	- \sigma_{ij}(\widetilde u(s)),\phi_1\rangle\big|^2 \,\textnormal{d}s \\
	&\le \|\phi_1\|_{L^2(\dom)}^2\widetilde\E\int_0^t\|\sigma_{ij}(\widetilde u^{(N)}(s))
	-\sigma_{ij}(\widetilde u(s))\|_{L^2(\dom)}^2 \,\textnormal{d}s \\
	&\le CT\|\phi_1\|_{L^2(\dom)}^2
	\widetilde\E\bigg(\sup_{0<s<T}\|\widetilde u^{(N)}(s)\|_{L^2(\dom)}^2
	+ \sup_{0<s<T}\|\widetilde u(s)\|_{L^2(\dom)}^2\bigg) \le C.
\end{align*}
This bound and convergence \eqref{2.convsig} allow us to apply the dominated
convergence theorem to conclude that for any $\phi_1\in L^2(\dom)$,
$$
  \lim_{N\to\infty}\widetilde\E\int_0^T\bigg|\bigg\langle\int_0^t
	\big(\sigma_{ij}(\widetilde u^{(N)}(s))-\sigma_{ij}(\widetilde u(s))\big)
	\textnormal{d}\widetilde W_j(s),\phi_1\bigg\rangle\bigg|^2 \,\textnormal{d}t = 0.
$$
This shows \eqref{3.E5} and finishes the proof.
\end{proof}

We define
\begin{align*}
  &\Lambda^{(N)}_i\big(\widetilde u^{(N)},\widetilde W^{(N)},\phi\big)(t)
	:= \big\langle\Pi_N(\widetilde u_i(0)),\phi\big\rangle \\
	&\phantom{xxxx}{}
	{-\sum_{j=1}^n\int_0^t\big\langle A_{ij}(\widetilde u^{(N)}(s))
	\na\widetilde u^{(N)}_j(s),\na\phi\big\rangle \,\textnormal{d}s} \\
	&\phantom{xxxx}{}+ \frac12\int_0^t\big\langle
	\Pi_N\mathcal{T}_i(\widetilde u^{(N)}(s)),\phi\big\rangle \,\textnormal{d}s 
	+ \sum_{j=1}^n\bigg\langle\int_0^t\Pi_N\sigma_{ij}(\widetilde u^{(N)}(s))
	\textnormal{d}\widetilde W_j^{(N)}(s),\phi\bigg\rangle, \\
	&\Lambda_i\big(\widetilde u,\widetilde W,\phi\big)(t)
	:= \langle\widetilde u_i(0),\phi\rangle
	{- \sum_{j=1}^n\int_0^t\big\langle A_{ij}(\widetilde u(s))
	\na\widetilde u_j(s),\na\phi\big\rangle \,\textnormal{d}s} \\
	&\phantom{xxxx}{} + \frac12\int_0^t \big\langle\mathcal{T}_i(\widetilde u(s)), 
	\phi \big\rangle \,\textnormal{d}s 
	+ \sum_{j=1}^n\bigg\langle\int_0^t\sigma_{ij}(\widetilde u(s))
	\textnormal{d}\widetilde W_j(s),\phi\bigg\rangle, 
\end{align*}
for $t\in[0,T]$ and $i=1,\ldots,n$. The following corollary is essentially 
a consequence of Lemma \ref{lem.E}; see \cite[Corollary 17]{DJZ19} for a proof.

\begin{corollary}\label{coro.E}
It holds for any $\phi_1\in L^2(\dom)$ and any $\phi_2\in H^3(\dom)$ 
satisfying $\na\phi_2\cdot\nu=0$ on $\pa\dom$ that
\begin{align*}
  \lim_{N\to\infty}\big\|\langle\widetilde u_{i}^{(N)},\phi_1\rangle
	- \langle\widetilde u_{i},\phi_1\rangle\big\|_{L^2(\widetilde\Omega\times(0,T))} &= 0, \\
  \lim_{N\to\infty}\big\|\Lambda_i^{(N)}\big(\widetilde u^{(N)},\widetilde W^{(N)},
	\phi_2\big)	- \Lambda_i\big(\widetilde u,\widetilde W,\phi_2\big)
	\big\|_{L^1(\widetilde\Omega\times(0,T))} &= 0.
\end{align*}
\end{corollary}

With these preparations, we can finish the proof of Theorem \ref{thm.ex}.
Indeed, since $u^{(N)}$ is a strong solution to \eqref{2.approx1}--\eqref{2.approx2},
it satisfies the identity
$$
  \langle u^{(N)}_i(t),\phi\rangle = \Lambda_i^{(N)}(u^{(N)},W,\phi)(t)
	\quad\Prob\mbox{-a.s.}
$$
for a.e.\ $t\in[0,T]$, $i=1,\ldots,n$, and $\phi\in H^1(\dom)$. In particular, it
follows that
$$
  \int_0^T\E\big|\langle u_i^{(N)}(t),\phi\rangle - \Lambda_i^{(N)}(u^{(N)},W,\phi)(t)
	\big|\,\textnormal{d}t = 0, \quad i=1,\ldots,n.
$$
Moreover, since the laws $\mathcal{L}(u^{(N)},W)$ and $\mathcal{L}(\widetilde u^{(N)},
\widetilde W^{(N)})$ coincide,
$$
  \int_0^T\widetilde\E\big|\langle\widetilde u_i^{(N)}(t),\phi\rangle
	- \Lambda_i^{(N)}(\widetilde u^{(N)},\widetilde W^{(N)},\phi)(t)\big|\,\textnormal{d}t = 0,
	\quad i=1,\ldots,n.
$$
We deduce from Corollary \ref{coro.E} that in the limit $N\to\infty$, this
equation becomes
$$
  \int_0^T\widetilde\E\big|\langle\widetilde u_i(t),\phi\rangle
	- \Lambda_i(\widetilde u,\widetilde W,\phi)(t)\big|\,\textnormal{d}t = 0, \quad i=1,\ldots,n.
$$
This identity holds for all $\phi\in H^3(\dom)$ satisfying $\na\phi\cdot\nu=0$
on $\pa\dom$ and, by density, also for all $\phi\in H^1(\dom)$. Hence,
for a.e.\ $t\in[0,T]$ and $\widetilde\Prob$-a.s.,
$$
  \big|\langle\widetilde u_i(t),\phi\rangle
	- \Lambda_i(\widetilde u,\widetilde W,\phi)(t)\big| = 0, \quad i=1,\ldots,n.
$$
The definition of $\Lambda_i$ implies that for a.e.\ $t\in[0,T]$
$\widetilde\Prob$-a.s.\ and for all $\phi\in H^1(\dom)$,
\begin{align*}
\langle \widetilde u_i(t),\phi\rangle &= \langle\widetilde u_i(0),\phi\rangle    { - \sum_{j=1}^n\int_0^t\big\langle A_{ij}(\widetilde u(s))
	\na\widetilde u_j(s),\na\phi\big\rangle \,\textnormal{d}s} \\
	&\phantom{xx}{}+ \frac12\int_0^t\langle\mathcal{T}_i(\widetilde u(s)),\phi
	\rangle \,\textnormal{d}s + \sum_{j=1}^n\bigg\langle\int_0^t\sigma_{ij}(\widetilde u(s))
	\textnormal{d}\widetilde W_j(s),\phi\bigg\rangle.
\end{align*}
Setting $\widetilde U=(\widetilde\Omega,\widetilde{\mathcal{F}},\widetilde{\mathbb{F}},
\widetilde\Prob)$, we deduce that $(\widetilde U,\widetilde u,\widetilde W)$ is a
martingale solution to \eqref{1.eq}--\eqref{1.bic}, and the stochastic process
$\widetilde u$ satisfies the estimates 
$$
  \widetilde\E\int_0^T\|\widetilde u(t)\|_{H^1(\dom)}^2 \,\textnormal{d}t<\infty,\quad
	\widetilde\E\bigg(\sup_{0<t<T}\|\widetilde u(t)\|_{L^\infty(\dom)}^p\bigg)<\infty
	\quad\mbox{for }p<\infty.
$$


\section{Examples}\label{sec.ex}

We present two examples that fulfill Assumptions (A3)--(A7). 

\subsection{Maxwell--Stefan systems}

Maxwell--Stefan equations describe the dynamics of fluid mixtures in the
diffusion regime. Applications include 
membrane electrolysis processes \cite{HVSKV01}, 
ion transport through nanopores \cite{BBN11}, 
and dynamics of lithium-ion batteries \cite{NBL08}.
Here, we consider an uncharged three-species mixture
with the concentrations $u_1$, $u_2$ and the solvent
concentration $u_3=1-u_1-u_2$. The diffusion matrix is given by
\begin{align*}
  & A(u) = \frac{1}{a(u)}\begin{pmatrix} 
	d_2+(d_0-d_2)u_1 & (d_0-d_1)u_1 \\ (d_0-d_2)u_2 & d_1+(d_0-d_1)u_2 \end{pmatrix}, \\
	&\mbox{where }a(u)=d_0d_1u_1 + d_0d_2u_2 + d_1d_2u_3,
\end{align*}
and $d_i>0$ for $i=0,1,2$ are diffusion coefficients \cite[Section 4.1]{Jue16}.
The matrix $A(u)$ is Lipschitz continuous on $\overline{\DD}$
since $a(u)$ is strictly positive and bounded from above (Assumption (A3)).
The entropy density is given by 
$$
  h(u) = \sum_{i=1}^3 {\big(u_i(\log u_i-1)+1\big)}.
$$
Its derivative $w=h'(u)=(\log(u_1/u_3),\log(u_2/u_3))^\top$ can be explicitly
inverted on $\DD$:
$$
  u_i = \frac{e^{w_i}}{1+e^{w_1}+e^{w_2}}, \quad i=1,2.
$$
Moreover, there exists $c>0$ such that for $z\in\R^n$,
$$
  z^\top h''(u)A(u)z = \frac{d_2z_1^2}{u_1a(u)} + \frac{d_1z_2^2}{u_2a(u)}
	+ \frac{d_0(z_1+z_2)^2}{u_3a(u)} 
	\ge c\bigg(\frac{z_1^2}{u_1} + \frac{z_2^2}{u_2}\bigg).
$$
Thus, Assumption (A5) is satisfied with $m=1/2$.

We choose the multiplicative noise 
$$
  \sigma(u) = \begin{pmatrix} u_1u_3 & 0 \\ 0 & u_2u_3 \end{pmatrix},
$$
where we recall that $u_3 = 1 - u_1 -u_2$. {This noise term guarantees that the
solutions stay in the Gibbs simplex a.s. Similar terms are well-known in
stochastic reaction-diffusion equations; see, e.g. \cite[(8)]{Kli11}.}
Then the expressions
\begin{align*}
  \bigg|\frac{\pa h}{\pa u_i}(u)\sigma_{ii}(u)\bigg| 
	&= \bigg|u_iu_3\log\frac{u_i}{u_3}\bigg|, \\
	\bigg|\sigma_{ii}(u)\frac{\pa\sigma_{ii}}{\pa u_i}(u)\frac{\pa h}{\pa u_i}(u)\bigg|
	&= \bigg|u_iu_3(u_3-u_i)\log\frac{u_i}{u_3}\bigg|, \\
	\bigg|\sigma_{ii}(u)\frac{\pa^2 h}{\pa u_i^2}(u)\sigma_{ii}(u)\bigg| 
	& = u_i u_3(u_i+u_3), \quad i=1,2,
\end{align*}
are bounded for $u\in\overline{\DD}$, proving Assumption (A6).
It remains to verify Assumption (A7). To simplify the notation, we set 
$u^\delta=(u_1^\delta,u_2^\delta)$ with $u_i^\delta:=[u_i]_\delta$.
and $u_3^\delta=1-u_1^\delta-u_2^\delta$.
We compute the elements $M^\delta_{ij}$ of the
matrix $h''(u^\delta)A(u)$:
\begin{align*}
  M^\delta_{11} &= \frac{1}{a(u)}\bigg(\frac{d_2}{u_1^\delta} 
	+ (d_0-d_2)\bigg(\frac{u_1}{u_1^\delta} - \frac{u_3}{u_3^\delta}\bigg)
	+ \frac{d_0}{u_3^\delta}\bigg), \\
	M^\delta_{12} &= \frac{1}{a(u)}\bigg((d_0-d_1)\bigg(\frac{u_1}{u_1^\delta} 
	- \frac{u_3}{u_3^\delta}\bigg) + \frac{d_0}{u_3^\delta}\bigg), \\
	M^\delta_{21} &= \frac{1}{a(u)}\bigg((d_0-d_2)\bigg(\frac{u_2}{u_2^\delta} 
	- \frac{u_3}{u_3^\delta}\bigg) + \frac{d_0}{u_3^\delta}\bigg), \\
	M^\delta_{22} &= \frac{1}{a(u)}\bigg(\frac{d_2}{u_1^\delta} 
	+ (d_0-d_1)\bigg(\frac{u_2}{u_2^\delta} - \frac{u_3}{u_3^\delta}\bigg)
	+ \frac{d_0}{u_3^\delta}\bigg).
\end{align*}
It holds for $z\in\R^n$ that
$$
  z^\top h''(u^\delta)A(u)z - c_h\sum_{i=1}^2\frac{z_i^2}{u_i^\delta}
	\ge z^\top R^\delta(u)z,
$$
where $c_h=\min\{d_0d_1,d_0d_2,d_1d_2\}>0$ and
\begin{align*}
  z^\top R^\delta(u)z &= \frac{d_0-d_2}{a(u)(1+\delta)^2}
	\bigg(\frac{u_1}{u_1^\delta}-\frac{u_3}{u_3^\delta}\bigg)z_1^2
	+ \frac{d_0-d_1}{a(u)(1+\delta)^2}
	\bigg(\frac{u_1}{u_1^\delta}-\frac{u_3}{u_3^\delta}\bigg)z_1z_2 \\
	&\phantom{xx}{}+ \frac{d_0-d_2}{a(u)(1+\delta)^2}
	\bigg(\frac{u_2}{u_2^\delta}-\frac{u_3}{u_3^\delta}\bigg)z_1z_2
	+ \frac{d_0-d_1}{a(u)(1+\delta)^2}
	\bigg(\frac{u_2}{u_2^\delta}-\frac{u_3}{u_3^\delta}\bigg)z_2^2.
\end{align*}
Since $u_i/u_i^\delta$ is bounded for $u\in\overline{\DD}$ and $i=1,2,3$, 
it follows that $R_\delta(u)\to 0$ as $\delta\to 0$ uniformly in $u\in\overline{\DD}$.
We infer that Assumption (A7) is fulfilled.


\subsection{Biofilm model}

Consider a fluid mixture consisting of $n$ concentrations $u_1,\ldots,u_n$ and
the solvent concentration $u_{n+1}$ such that $\sum_{i=1}^{n+1} u_i=1$. We suppose that
the concentrations are driven by the partial pressures $p_i=u_i$ ($i=1,\ldots,n$), 
while the solvent has the constant partial pressure $p_{n+1}$. 
Allowing for the presence of an interphase force and neglecting inertia effects,
a volume-filling cross-diffusion model with diffusion matrix $A(u)$, defined by
$$
  A_{ii}(u)=1-u_i, \quad A_{ij}(u) = -u_i\quad\mbox{for }i\neq j,
$$
was formally derived in \cite[Example 4.3]{Jue16} from an Euler system with
linear friction force. This model can be also used to describe the dynamics of 
a bacterial biofilm with subpopulations $u_1,\ldots,u_n$ and the volume fraction 
$u_{n+1}$ of ``free space'', in which the biofilm can expand \cite{Dav18}.
As in the previous example, we choose the entropy density and the noise term 
$$
  h(u)=\sum_{i=1}^{n+1} {\big(u_i(\log u_i-1)+1\big)}, \quad
	\sigma_{ii}(u) = u_iu_{n+1}, \quad \sigma_{ij}(u)=0\quad\mbox{for }i\neq j.
$$
The previous example has shown that Assumption (A6) is satisfied.
Assumption (A5) is fulfilled with $m=1/2$ since for all $u\in\DD$ and $z\in\R^n$,
$$
  z^\top h''(u)A(u)z = \sum_{i=1}^n\frac{z_i^2}{u_i}.
$$
It remains to check Assumption (A7). For this, we compute 
\begin{align*}
  & z^\top h''(u^\delta)A(u)z - \sum_{i=1}^n\frac{z_i^2}{u_i^\delta}
	= z^\top R_\delta(u)z, \quad\mbox{where} \\
  & R_\delta(u) = \sum_{i,j=1}^n
	\bigg(\frac{u_{n+1}}{u_{n+1}^\delta}-\frac{u_i}{u_i^\delta}\bigg)z_iz_j.
\end{align*}
It holds that $R_\delta(u)\to 0$ as $\delta\to 0$ uniformly in $u\in\overline\DD$.\\

\textbf{Acknowledgments.} The authors thank the referee for numerous suggestions and are grateful to Prof. Michael R\"ockner for very helpful comments. The first three authors acknowledge partial support from   
the Austrian Science Fund (FWF), grants I3401, P30000, W1245, and F65. The last two authors have been supported by a German Science Foundation (DFG) grant in the D-A-CH framework, grant KU 3333/2-1. The fourth author acknowledges partial support by a Lichtenberg Professorship funded by the VolkswagenStiftung and by the Collaborative Research Center 109 ``Discretization in Geometry and Dynamics'' funded by the DFG.
\begin{appendix}
\section{Technical results}\label{sec.tech}

For the convenience of the reader, we recall some technical results used
in this paper. Since we are working on the non-metric space $Z_T$, we need
Jakubowski's generalization of the Skorokhod theorem
in the form given in \cite[Theorem C.1]{BrOn10} (see \cite{Jak97} for 
the original theorem).

\begin{theorem}[Skorokhod--Jakubowski]\label{thm.skoro}
Let $Z$ be a topological space such that there exists a sequence $(f_m)_{m\in\N}$
of continuous functions $f_m:Z\to\R$ that separate points of $Z$. Let $S$ be the
$\sigma$-algebra generated by $(f_m)_{m\in\N}$. Then
\begin{enumerate}
\item Every compact subset of $Z$ is metrizable.
\item If $(\mu_m)_{m\in\N}$ is a tight sequence of probability measures on $(Z,S)$,
then there exists a subsequence $(\mu_{m_k})_{k\in\N}$, a probability space
$(\widetilde\Omega,\widetilde{\mathbb{F}},\widetilde\Prob)$, 
and $Z$-valued Borel measurable
random variables $\xi_k$ and $\xi$ such that {\rm (i)} $\mu_{m_k}$ is the law of $\xi_k$
and {\em (ii)} $\xi_k\to\xi$ almost surely on $\widetilde\Omega$.
\end{enumerate}
\end{theorem}

The following result is proved in \cite{Kur66} (also see \cite[Theorem B2]{BrDh19}).

\begin{theorem}[Kuratowski]\label{thm.kura}
Let $X$ be a separable complete metric space, $Y$ a Borel set of $X$, and
$f:Y\to X$ a one-to-one Borel measurable mapping. Then for any Borel set
$B\subset Y$, the image $f(B)$ is a Borel set.
\end{theorem}

The Wong--Zakai approximations converge to the Wiener process. This was proved in
\cite{WoZa65} in the one-dimensional case, extended in \cite{StVa72} to
higher dimensions, and unified in \cite[Chapter 6, Theorem 7.2]{IkWa89}.

\begin{theorem}[Convergence of Wong--Zakai approximations]\label{thm.wong} 
Let $X^{(\eta)}$ be the solutions to the family of ODEs, indexed by the random
variable $\omega\in\Omega$, on a finite-dimensional vector space $H$,
$$
  \textnormal{d}X^{(\eta)}(t) = a(X^{(\eta)}(t),t)\,\textnormal{d}t + b(X^{(\eta)}(t),t)\,\textnormal{d}W^{(\eta)}(t), \
	t\in[0,T], \quad X^{(\eta)}(0)=X^0,
$$
where $W^{(\eta)}$ are the 
Wong--Zakai approximations \eqref{2.wong} of a Wiener process with time step
$\eta>0$; $a(X,\cdot)$, $b(X,\cdot)$, $(\pa b/\pa t)(X,\cdot)$, and
$(\pa b/\pa X)(X,\cdot)$ are continuous; and $a(\cdot,t)$, $b(\cdot,t)$, and
$(\pa b/\pa X)(\cdot,t)$ are Lipschitz continuous (and consequently grow at most
linearly). Furthermore, let $X$ be a solution to the Stratonovich stochastic
differential equation
$$
  \textnormal{d}X(t) = a(X(t),t)\,\textnormal{d}t + b(X(t),t)\circ \,\textnormal{d}W(t), \ t\in[0,T], \quad
	X(0)=X^0.
$$
Then
$$
  \lim_{\eta\to 0}\E\bigg(\sup_{0<t<T}\|X^{(\eta)}(t)-X(t)\|_H^2\bigg) = 0.
$$
\end{theorem}
\end{appendix}


\end{document}